\documentclass{cmslatex}
\usepackage{amsmath, latexsym, amsfonts,bm,amssymb,mathscinet} 

\usepackage[paperwidth=7in, paperheight=10in, margin=.875in]{geometry}
 \usepackage[colorlinks,linkcolor=red,anchorcolor=green,citecolor=blue]{hyperref}
\renewcommand{\theequation}{\arabic{section}.\arabic{equation}}

   \allowdisplaybreaks

\usepackage{appendix}
\usepackage{natbib} 

\usepackage{url} 
\usepackage{enumerate}
\usepackage{mathtools} 
\usepackage[dvipsnames]{xcolor} 
\usepackage{mathtools} 

\usepackage[ngerman,english]{babel}
\useshorthands{"}
\addto\extrasenglish{\languageshorthands{ngerman}} 
\usepackage{booktabs} 
 
\usepackage{ wasysym } 

\usepackage{enumitem} 

\bibpunct{(}{)}{;}{a}{}{,} 

\newtheorem{cnd}{Condition}

\newcommand{\pp}{\mathbb{P}}

\newcommand{\QQ}{\mathbb{Q}}

\newcommand{\cf}{\mathcal{F}}

\newcommand{\RR}{\mathbb{R}}
\newcommand{\dd}{\mathrm{d}}
\newcommand{\eps}{\varepsilon}

\newcommand{\one}{\mathbf{1}}
\newcommand{\half}{\frac{1}{2}}

\allowdisplaybreaks

\usepackage[OT2,OT1]{fontenc}
\newcommand\cyr{%
\renewcommand\rmdefault{wncyr}%
\renewcommand\sfdefault{wncyss}%
\renewcommand\encodingdefault{OT2}%
\normalfont
\selectfont}
\DeclareTextFontCommand{\textcyr}{\cyr}

\newcommand{\Ga}{\operatorname{Gamma}}
\renewcommand{\tilde}{\widetilde} 

\allowdisplaybreaks

\newcommand{\cH}{{\mathcal{H}}}

\newcommand{\ii}{\mathrm{i}}

\begin{document}
 \title{Nonparametric Bayesian inference for Gamma-type L\'evy subordinators\thanks{Received date, and accepted date (The correct dates will be entered by the editor).}}


\author{Denis Belomestny\thanks{Duisburg-Essen University,
Thea-Leymann-Str.~9,
D-45127 Essen,
Germany,
and
National Research University, 
Higher School of Economics, Moscow, Russian Federation,
\url{denis.belomestny@uni-due.de}}
\and
Shota Gugushvili\thanks{Biometris,
		Wageningen University \& Research,
		Postbus 16,
		6700 AA Wageningen,
		The Netherlands,
\url{gugushvili@gmail.com}}
\and
Moritz Schauer\thanks{Mathematical Institute,
Leiden University,
P.O. Box 9512,
2300 RA Leiden,
The Netherlands,
\url{m.r.schauer@math.leidenuniv.nl}}
\and
Peter Spreij\thanks{Korteweg-de Vries Institute for Mathematics,
University of Amsterdam,
P.O. Box 94248,
1090 GE Amsterdam,
The Netherlands,
 and Institute for Mathematics, Astrophysics and Particle Physics, Radboud University, Nijmegen, The Netherlands,
\url{spreij@uva.nl}}
}

         \pagestyle{myheadings} \markboth{Bayesian inference for subordinators}{Belomestny, Gugushvili, Schauer, Spreij} \maketitle

          \begin{abstract}
Given discrete time observations over a growing time interval, we consider a nonparametric Bayesian approach to estimation of the L\'evy density of a L\'evy process belonging to a flexible class of infinite activity subordinators. Posterior inference is performed via MCMC, and we circumvent the problem of the intractable likelihood  via the data augmentation device, that in our case relies on bridge process sampling via Gamma process bridges. Our approach also requires the use of a new infinite-dimensional form of a reversible jump MCMC algorithm.  We show that our method leads to good practical results in challenging simulation examples. On the theoretical side, we establish that our nonparametric Bayesian procedure is consistent: in the low frequency data setting, with equispaced in time observations and intervals between successive observations remaining fixed, the posterior asymptotically,  as the sample size $n\rightarrow\infty$, concentrates around the L\'evy density under which the data have been generated. Finally, we test our method on a classical insurance dataset.
          \end{abstract}


\section{Introduction}
\label{sec:intro}
In this paper, to the best of our knowledge for the first time in the literature, we study the problem of nonparametric Bayesian inference for infinite activity subordinators, i.e., L\'evy processes with non-decreasing sample paths.
In the last two decades, L\'evy processes have received a lot of attention, mainly due to
their numerous applications in mathematical finance and insurance, but also in natural sciences; see, e.g., \cite{levy01}. As a matter of fact, thanks to their ability to reproduce stylised features of financial time series distributions, L\'evy processes have become a fundamental building block for modelling asset prices with jumps, see \cite{tankov2003financial}. By the L\'evy-Khintchine formula, the law of a L\'evy process is uniquely determined by the so-called L\'evy triplet, which hence encodes all the probabilistic information on the process. Since the L\'evy triplet involves an infinite-dimensional object, the L\'evy measure of the process, this provides natural motivation for studying nonparametric inference procedures for L\'evy processes, where the objects of inference are elements of some function spaces.

We term the class of increasing infinite activity L\'evy processes that we study $\theta$-subordinators. Our model generalises the well-known Gamma process, which is a popular risk model, see \cite{dufresne91}, and also forms a building block for more general L\'evy models, like the Variance-Gamma (VG) process, that finds many applications in finance, see, e.g., \cite{madan1990variance}. The family of $\theta$-subordinators also overlaps with the class of self-decomposable L\'evy processes, that likewise have important applications in finance, see, e.g., \cite{carr07}.

We specifically concentrate on estimation of the L\'evy triplet of a $\theta$-subordinator. On the computational side, our Bayesian procedure circumvents the problem of the intractable likelihood for $\theta$-subordinators via the data augmentation device, which relies on bridge process sampling via Gamma process bridges, and also employs an infinite-dimensional form of the reversible jump algorithm. On the theoretical side, we establish that our procedure is consistent: as the sample size grows to infinity, the posterior asymptotically concentrates around the parameters of the L\'evy processes under which the data have been generated. We test our algorithm on simulated and real data examples. In particular we fit a $\theta$-subordinator to a benchmark dataset in insurance theory, large fire losses in Denmark, and study the question whether a risk model based on a Gamma process is adequate for modelling this dataset.

\subsection{Literature overview}

To provide further motivation for a nonparametric Bayesian approach to inference in L\'evy processes and to highlight some associated challenges, in this subsection we supply an overview of the literature on the subject.

The problem of nonparametric  inference for L\'evy processes has a long history, going back to  \cite{RT} and \cite{BB}. Revival of interest in it dates around the year 2003, with contributions \cite{buchmann03}, \cite{buchmann04} and \cite{gugu07}, as well as numerous later publications; see also \cite{ilhe15} for a further extension. Very recent works \cite{coca17} and \cite{duval17} provide an extensive list of references.

In general, there are two major strands of mathematical statistics literature dealing with inference for L\'evy processes, or more generally semimartingales. The first considers the so-called high frequency setup where asymptotic properties of the corresponding estimators are studied under the assumption that observations are made at an increasing frequency in time.  In the second strand of the literature, times between successive observations are assumed to be fixed (the so-called low frequency setup) and the asymptotic analysis is done under the premise that the observational horizon tends to infinity.

The last decade  witnessed a tremendous advance  in the area of statistics for high frequency financial data, due to the development of new mathematical methods to analyse these data, as well as increasing availability of such data. We refer to the recent book \cite{ait2014high} for a comprehensive treatment of modern statistical methods for high frequency data. At the same time,  progress was achieved also in statistical inference for L\'evy-driven models based on low frequency data, see, e.g., \cite{belomestny2015estimation} for an overview and references. The latter situation is more challenging, as e.g. it becomes quite difficult to  distinguish between small jumps of a L\'evy process  and the Brownian increments. This often leads  to rather slow, logarithmic convergence rates for resulting estimators, see, e.g., \cite{belomestnyReiss2006}, \cite{gugu09} and \cite{gugushvili2012nonparametric}. Hence accurate nonparametric inference for L\'evy processes typically requires very large amounts of data, which may not always be available in practice. Fortunately, in many cases there is additional (prior) information about the structure of the parameters, which can be used to improve the estimation quality under limited data. To  account for this prior information, the Bayesian estimation framework is quite appealing. Furthermore, the Bayesian approach provides automatic uncertainty quantification in parameter estimates through the spread of the posterior distribution of the parameters. Also, in some fields, such as e.g.~climate and weather science, Bayesian approaches are thought to be default (see, e.g., \cite{berliner99}), and studying them would go together with common practices in those fields. On the other hand, there are also some formidable challenges in applying the nonparametric Bayesian methodology to inference in L\'evy processes. Firstly, the underlying process is usually observed at discrete time instances, while L\'evy models are formulated in continuous time. This gives rise to complications that are typical in inference for discretely observed continuous time stochastic processes. Secondly, Bayesian estimation in its simplest, pristine form requires knowledge of the likelihood of observations, and hence of marginal densities of the underlying L\'evy process; these, however, are rarely available in closed form. Thirdly, devising valid MCMC algorithms in infinite-dimensional settings is a highly non-trivial task. Cf.\ recent works on nonparametric Bayesian inference in diffusion models, such as \cite{beskos08} and \cite{frank14}.
\par
The literature on nonparametric Bayesian inference for L\'evy processes is very recent and also rather scarce, the only available works being \cite{nickl17}, \cite{cpp15} and \cite{cpp16}. These deal with a particular case of compound Poisson processes, concentrate exclusively on theoretical aspects (with the exception of the latter paper), and do not appear to admit an obvious extension to other classes of L\'evy processes. In fact, compound Poisson processes are rather special among L\'evy processes, and are of limited applicability in many practically relevant cases. Hence there is space for improvement. On the positive side, the practical results we obtained in this work demonstrate great potential of Bayesian methods for inference in L\'evy processes. Our approach is aimed at developing an applicable statistical methodology, which we substantiate by theoretical results, and also test via challenging examples. At the same time, we admit there remain several unresolved theoretical and practical issues, such as derivation of posterior contraction rates or practical fine-tuning of the prior we use. However, upon careful reading this should come as no surprise given the sheer complexity of our undertaking, where several topics would have merited to be subjects of independent research projects. We view our work as the first substantial step made in the direction of studying inference problems for L\'evy processes via nonparametric Bayesian methods. It is our hope that our contribution will generate additional interest in this statistically and mathematically fascinating topic.
\subsection{Structure of the paper}
The rest of the paper is organised as follows: in Section \ref{sec:likelihood} we describe in detail the statistical problem we are dealing with and our nonparametric Bayesian approach to it. Posterior inference in our setting is performed through MCMC sampling, and Section \ref{sec:posterior} provides a detailed exposition of our sampling algorithm. In Section \ref{section:consistency} we establish the fact that our approach is consistent in the frequentist sense: asymptotically, as the sample size $n\rightarrow\infty$, the posterior measure concentrates around the L\'evy triplet under which the data used in the estimation procedure has been generated. In Section \ref{sec:example2} we test the practical performance of our method via simulation on a challenging example. In Section \ref{sec:beta} we further generalise our basic model from Section \ref{sec:likelihood} and detail changes and extensions this requires in designing an MCMC sampler in comparison to the one from Section \ref{sec:posterior}. This new sampler is tested in simulations in Section \ref{sec:example:beta}. In Section \ref{sec:danish} we apply our methodology on an insurance dataset. Possible extensions of our inferential approach to more general L\'evy models are discussed in Section \ref{sec:outlook}. Finally, in Appendices \ref{lemmata0} and \ref{lemmata} we state and prove some technical results used in the main body of the paper, while in Appendix~\ref{app:danish} we provide some additional analyses to substantiate our modelling approach in Section~\ref{sec:danish}.

\section{Statistical problem and approach}
\label{sec:likelihood}

In this section we introduce in detail the statistical problem we are dealing with and describe our approach to tackle it.

\subsection{Statistical problem}
\label{subsec:intro}
{
Consider a univariate L\'evy process $X=(X_t\colon t\geq 0)$ with generating L\'evy triplet $(\gamma,0,\nu)$, where $\nu([1, \infty))$ is finite and 
\begin{equation}
\label{cond.gamma}
\gamma=\int_{0}^{1} x\nu(\dd x) <\infty.
\end{equation}
Hence $X$ has no Gaussian component and the law $\pp_\nu$ of \(X\) is entirely determined by $\nu$. By the L\'evy-Khintchine formula, see Theorem 8.1 in \cite{sato99}, the characteristic function $\phi_{X_1}$ of $X_1$ admits the unique representation of the type
\[
\phi_{X_1}(z)=\exp\left( \mathrm{i} \gamma z+\int_{\mathbb{R}} \left( e^{\mathrm{i}zx}-1-\mathrm{i}zx\one_{|x|\leq 1} \right) \nu (\dd x) \right).
\]
 We also assume that the L\'evy measure \(\nu\) admits the representation
\begin{equation}\label{levy}
\nu(\dd x)=\frac{\beta}{x}e^{-\alpha x-\theta(x)} \dd x, \quad x > 0,
\end{equation}
where $\alpha$ and $\theta\colon[0,\infty)\rightarrow \mathbb{R}$ are parameters to be estimated, while $\beta$ is a known or unknown parameter. 
It follows that  $X$ is a pure jump process with non-decreasing sample paths,
or put another way, a subordinator with zero drift, cf.\ Sections 2.6.1--2.6.2 in \cite{kyprianou}. 
One may call this class of L\'evy processes Gamma-type subordinators,
because $X$ is a Gamma process when $\theta \equiv 0$, but we prefer to simply refer to it as $\theta$-subordinators.
}
\par
Assume that the process $X$ is observed at discrete time instances $0 = t_0 <t_1 <\dots < t_n = T,$ so our observations are $X^{(n)} = (X_{t_i} \colon  i \in \{0,\ldots, n\})$. Our aim is nonparametric Bayesian estimation for the parameter triple $(\alpha,\beta,\theta)$. This requires specification of the likelihood and the prior in our model, that are next combined via Bayes' formula to form the posterior distribution. This latter encodes all the necessary inferential information within the Bayesian setup. By Theorem 27.7 in \cite{sato99}, marginal distributions of $X$ possess densities with respect to the Lebesgue measure.  With $p_h(x; \beta, \alpha, \theta)$ denoting the density of an increment $X_{t+h}-X_t$, the likelihood 
\[
\prod_{i=1}^n p_{t_i - t_{i-1}}(X_{t_i} - X_{t_{i-1}}; \beta, \alpha, \theta)
\]
is in general intractable, as the marginal densities of $X$ are not known in closed form, except some special cases. This complicates a computational approach to Bayesian inference. We will circumvent this obstacle by employing the concept of data augmentation, see \cite{TannerWong}. Specifically, we will propose a suitable nonparametric  prior distribution $\pi(\beta, \alpha, \theta)$ on the parameter triple $(\beta, \alpha,\theta)$, and derive a Metropolis-Hastings algorithm relying on data augmentation to sample from the posterior distribution. Details of our approach are given in the following subsections.

\subsection{Likelihood}
\label{subsec:likelihood}
We first consider the problem where $\beta$ is known and fixed. All processes and their laws in this section are restricted to the time interval $[0, T]$ for a fixed $T > 0$.
Note that for any two L\'evy measures $\nu$ and $\nu_0$ given by \eqref{levy} with parameters $\beta,\alpha, \theta$ 
and $\beta,\alpha_0,\theta_0,$ respectively, provided $\theta(0) = \theta_0(0) = 0$ and both functions $\theta$ and $\theta_0$ are Lipschitz in some neighbourhood of zero, we have 
\begin{equation}
\label{cond.nu}
\begin{split}
\nu & \text{ and }  \nu_0 \text{ are equivalent,}\\
d^2_\cH(\nu, \nu_0)&=\frac{1}{2}\int\limits_{(0,\infty)} ( \sqrt{\dd \nu} - \sqrt{\dd \nu_0} )^2<\infty,
\end{split}
\end{equation}
where $d_\cH(\cdot,\cdot)$ is the Hellinger distance between two (infinite) measures. By assumption \eqref{cond.gamma} and property \eqref{cond.nu}, together with Theorem 33.1 in \cite{sato99}, it follows that the laws $\pp_{\nu}$ and $\pp_{\nu_0}$ of $X =(X_t\colon t\in [0,T])$ are equivalent. Furthermore, Theorem 33.2 in \cite{sato99} implies that a.s.\
\begin{equation*}
U_T=\log\left( \frac{\dd \pp_{\nu}}{\dd \pp_{\nu_0}}\big(X\big)\right) = \sum_{(s,\Delta X_s)\in (0,T]\times\{ \Delta X_s>0 \}} \phi(\Delta X_s) - T \int\limits_{\mathclap{(0,\infty)}} (e^{\phi(x)}-1)\nu_0(\dd x) ,
\end{equation*}
where $\Delta X_s=X_s-X_{s-},$ and
\[
\phi(x)=\log \left(\frac{\dd \nu }{\dd \nu_0}(x)\right) = -(\alpha x + \theta(x) - \alpha_0 x - \theta_0(x)),  \quad x>0.
\]
We can also write the log-likelihood ratio $U_T$ as
\[
U_T=\int_{(0,T]} \int_{{(0,\infty)}} \phi(x) \mu(\dd s,\dd x) - T\int\limits_{\mathclap{(0,\infty)}} (\nu-\nu_0)(\dd x),
\]
where the jump measure $\mu$ is defined by
\[
\mu((0,t]\times B)= \#\left\{ s\colon (s,\Delta X_s)\in (0,t]\times B\right\}
\]
for any Borel subset $B$ of $(0,\infty)$.
We can view  $\pp_{\nu_0}$ as the dominating measure for $\pp_{\nu}$. From the inferential point of view the specific choice of the dominating measure is immaterial.  A convenient choice of $\nu_0$ for the theoretical development in Section \ref{section:consistency} is to actually take $\nu_0$ to be the `true' L\'evy measure $\nu_0$ with parameters $\alpha_0$ and $\theta_0$ (recall that $\beta$ is fixed and assumed to be known).

\subsection{Gamma processes}

We temporarily specialise to the case of a Gamma process.
A Gamma process is an example of a pure jump L\'evy process with non-decreasing sample paths. Its L\'evy triplet is given by $(\gamma,0,\nu),$ where
\[
\gamma =\int_{0}^1 x\nu(\dd x), \quad \nu(\dd x)=\frac{\beta}{x}e^{-\alpha x} \dd x, \quad x>0,
\]
see Example 8.10 in \cite{sato99}. Making the dependence on parameters explicit, we also refer to $X$ as a $\Ga(\beta, \alpha)$ process. The distribution of $X_t,$ $t\in[0,T],$ is gamma with rate parameter $\alpha$ and shape parameter $\beta t,$ so that
\begin{equation}
\label{marginal.d}
X_t \sim p_{t}(x;\beta,\alpha) = \frac{\alpha^{\beta t} x^{\beta t-1} e^{-\alpha x}}{\Gamma(\beta t)}, \quad x>0,
\end{equation}
where $\Gamma$ denotes the gamma function.

\subsection{Data augmentation and bridge sampling}

By using the data augmentation technique, we can utilise existence of a closed-form likelihood for a continuously observed L\'evy path, see Subsection \ref{subsec:likelihood}, to define a Metropolis-Hastings algorithm to sample from the posterior given the discrete observations $X^{(n)}$. This treats the unobserved path segments between two consecutive observation times as missing data and augments the state space of the algorithm to sample from the joint posterior of missing data and unknown parameters.
Specifically, this requires the ability to sample from the conditional distribution of the missing data given the parameters and the observations. 

Consider again the L\'evy process $X =  (X_t\colon t \in [0,T])$  with fixed parameters $\beta$, $\alpha$, $\theta$, and denote the corresponding law by $\pp$.
Conditional on the observations $X_{t_{i-1}}$ and $X_{t_i}$ and the parameters, by the independent increments property of a L\'evy process, the process can be sampled on each time interval $[t_{i-1}, t_i]$ independently. Samples from the conditional distribution on these intervals connect the observations in the form of so-called bridges.
It suffices to describe the construction for a single bridge from $0$ to $T$.
A Gamma process $\tilde X  =  (\tilde X_t\colon t \in [0,T])$ shares with the Wiener process a remarkable property that samples from the conditional distribution can be obtained through a simple transformation of the unconditional path, see  \cite{yor07}. For the Wiener process $W$ conditional on $W_T = w_T$ for a number $w_T$, this transformation takes the form 
\[
t \mapsto W_t + \frac{t}{T}(w_T - W_T), \quad t \in [0, T].
\]
For the Gamma process, the corresponding transformation takes a multiplicative form:
define for a path $X =  (X_t\colon t \in [0,T])$  a map $g_{x_T}$ by
\begin{equation}\label{forward}
g_{x_T}( X) = (x_T { X_t}/{X_{T} }\colon t \in [0,T]).
\end{equation}
Then $\tilde \pp^\star = g_{x_T} \circ\, {\tilde \pp} $,
where $\tilde\pp$ denotes the law of $\tilde X$,
 defines a factorisation of the conditional distribution $\tilde \pp^\star$ of $\tilde X$ under the law $\tilde \pp$ given $\tilde X_T = x_T$.
This result in combination with a Metropolis-Hastings step can be used to sample from the conditional distribution of a $\theta$-subordinator given the observations and parameters.


Analogously, we denote by $\pp^\star$ the conditional distribution of $X$ under the law $\pp$ given $X_T = x_T$.
Here and later we use a superscript star to denote the conditional distributions, suppress the dependence on $x_T$ in the notation and write for example  $\pp^\star(\dd X)$ for integration with respect to the conditional distribution.
By conditioning, 
\begin{equation}
\label{formula:bayes}
 \frac{{\dd\pp^\star}}{\dd \tilde \pp ^\star} (g_{x_T}(X)) = \frac{\tilde p(x_T) }{p(x_T) }  \frac{\dd\pp}{\dd\tilde \pp}(g_{x_T}(X)),
\end{equation}
where $p$ and $\tilde p$ are the densities of $X_T$ under $\pp$ and $\tilde \pp$, respectively. 
 Note that \(\frac{\dd\pp}{\dd\tilde \pp}(g_{x_T}(X))\)  is the continuous-time likelihood, which is known in closed form. Hence $\frac{\dd \pp^\star}{\dd\tilde \pp^\star}$ is also known in closed form up to an unknown proportionality constant \(\frac{\tilde p(x_T) }{p(x_T) }\), and the ratio of Radon-Nikodym derivatives $ \frac{\dd \pp^\star}{\dd\tilde \pp^\star}(X^\circ) \big/ \frac{\dd \pp^\star}{\dd\tilde \pp^\star}(X)$, with $X^{\circ}$ denoting a proposal in the MCMC algorithm, is given by formula \eqref{formula:pathlr} below. This allows us to use samples distributed according to  $\tilde\pp^\star$, i.e.~$\Ga(\beta, \alpha)$ bridges, as proposals for the augmented segment that 
follows the intractable conditional distribution $\pp^\star$.

\subsection{Prior}
\label{subsec:prior}
To define the prior, we consider a subclass of processes defined in  \eqref{levy}, where
the parameter $\theta$ in the L\'evy measure $\nu$ has the following form. Fix a sequence
\[
0 <b_1<\cdots<b_{N}<\infty,
\]
set for convenience  $b_{0} = 0$ and $b_{N+1}=\infty,$ and define bins $B_k$ by
\[
B_k=[b_{k},b_{k+1}), \quad k=0,\ldots, N.
\]
Given bins $B_k,$ assume the function $\theta$ is piecewise linear, i.e.,
\begin{equation}
\label{piecewise}
\theta(x)=\sum_{k=1}^N (\rho_k+\theta_k x) \one_{B_k},
\end{equation}
where $\rho_k\in\mathbb{R},$ $k=1,\ldots,N,$ $\theta_k\in\mathbb{R},$ $k=1,\ldots,N,$ and $\theta_N>-\alpha.$ Together with $\alpha,$ the parameter $\theta_k$ determines the slope of the function $\theta(x) + \alpha x$ on the bin $B_k,$ while $\rho_k$ gives the intercept. The process $X$ with the law $\pp_{\nu}$ can be viewed as a Gamma process with rate parameter $\alpha$ and shape parameter $\beta$,  subjected to local deviations in the behaviour
of jumps of sizes falling in bins $B_k$ compared to what of a Gamma process. The parameters $\theta_k,\rho_k$ quantify the extent of these local deviations on the bin $B_k$.
\par
We equip $\alpha,\theta_k,\rho_k$ with independent priors. Note that these priors on $\alpha,\theta_k,\rho_k$ implicitly define a prior on the L\'evy measure $\nu$ as well.
The specific form of the prior is not crucial for many arguments that follow, but  is convenient computationally. In fact, theoretical results in Section~\ref{section:consistency} can be derived for other series priors as well. However, the local linear structure in \eqref{piecewise} (which also means that  the prior could be rewritten as series prior with basis functions with compact support) is important to derive some simple update formulae below.
\par
For a realisation $\nu$ from the implicit prior on $\nu$ as above in the present section, let us work out the integral
\[
\nu(B_k)=\int_{b_{k}}^{b_{k+1}} \frac{\beta}{x} e^{-(\alpha + \theta_k) x-\rho_k}\dd x,
\]
which enters the expression for the likelihood in Subsection \ref{subsec:likelihood}. To that end remember the definition of the exponential integral,
$
E_1(z)=\int_z^{\infty} t^{-1}{e^{-t}}\dd t,
$
see, e.g., \S 15.09 in \cite{jeffreys99} for its basic properties. Then a change of the integration variable gives
\begin{equation}
\label{nu.bk}
\nu(B_k)=\beta e^{-\rho_k} \{ E_1((\theta_k + \alpha) b_{k}) - E_1((\theta_k + \alpha) b_{k+1}) \}, \quad k=1,\ldots,N.
\end{equation}
Observe that $\nu(B_N)=\beta e^{-\rho_N} E_1((\theta_k + \alpha) b_{N}).$ Similar to the case of $\nu$,
\[
\nu_0(B_k)=\beta\{ E_1(\alpha_0 \,b_{k}) - E_1(\alpha_0 \,b_{k+1}) \}, \quad k=1,\ldots,N.
\]
Also here remark that $\nu_0(B_N)=\beta E_1(\alpha_0 \,b_{N}).$ For future reference in Subsection \ref{subsec:likelihood:cases}, note that for any $\alpha,\alpha^{\prime},$
\begin{equation}
\label{frullani}
\lim_{x \to 0} \{ E_1(\alpha x) -  E_1(\alpha' x) \}= 
\log\left( \frac{\alpha'}{\alpha }\right),
\end{equation}
which follows from the formula for Frullani's integral, see \S 12.16 in \cite{jeffreys99}.

\subsection{Likelihood expressions for parameter updates}
\label{subsec:likelihood:cases}

The following expressions will be used in Section \ref{sec:posterior} to construct the Metropolis-Hastings algorithm to sample from the posterior of $\alpha,\theta_k,\rho_k.$
Define random variables
\[
\mu_T(B_k)=\mu((0,T]\times B_k)=\#\{ s\colon (s,\Delta X_s)\in (0,T]\times B_k \},
\]
that for each $k=1,\ldots,N,$ give the number of jumps of $X$, whose sizes fall into the bin $B_k.$ Consider two laws $\pp_{\nu}$ and $\pp_{\nu^{\circ}}$, where the L\'evy measure $\nu$ is given by \eqref{levy} and \eqref{piecewise}, while $\nu^\circ$ is given by \eqref{levy} with coefficients $\alpha^\circ,$ $\theta^\circ_1, \dots, \theta^\circ_N,$ $\rho^\circ_1, \dots, \rho^\circ_N$ instead of the coefficients $\alpha,$ $\theta_1, \dots, \theta_N$, $\rho_1, \dots, \rho_N$. The two laws $\pp_{\nu}$ and $\pp_{\nu^{\circ}}$ are equivalent, since each is equivalent to $\pp_{\nu_0}$. 
We have the following expression for the log-likelihood,
\begin{equation}\label{loglikeli}
\begin{split}
\log \frac{\dd \pp_{\nu^\circ}}{\dd \pp_\nu}(X) ={}&  -(\alpha^\circ - \alpha)\sum_{\mathclap{\substack{\Delta X_s \in B_0,\\ 0<s\leq T}}} \Delta X_s
-\sum_{k=1}^N (\theta^\circ_k + \alpha^\circ -\theta_k - \alpha) \sum_{\mathclap{\substack{\Delta X_s \in B_k,\\ 0<s\leq T}}} \Delta X_s\\
& -\sum_{k=1}^N (\rho_k^{\circ}-\rho_k) \mu_T(B_k) - T \sum_{k=0}^N (\nu^\circ-\nu)(B_k),
\end{split}
\end{equation}
where $\nu(B_k),k=1,\ldots,N,$ can be evaluated using \eqref{nu.bk}, and an analogous formula holds for $\nu^\circ(B_k)$,
whereas by \eqref{frullani} 
\[
(\nu^{\circ}-\nu)(B_0)=\beta \log\left( \frac{\alpha}{\alpha^{\circ} } \right)-\beta\{ E_1(\alpha^{\circ}b_1)-E_1(\alpha b_1) \}.
\]

Finally, for the ratio of Radon-Nikodym derivatives with respect to the law of a Gamma process $ \pp_{\tilde \nu}$ with the same parameter $\beta$ we have
\begin{equation}
\label{formula:pathlr}
\begin{split}
\log\left( \frac{\frac{\dd\pp_{\nu} }{\dd  \pp_{\tilde \nu}} (X^{\circ})}
 {\frac{\dd\pp_\nu}{\dd \pp_{\tilde \nu}} (X)}\right)
  ={}&
   -\sum_{k=1}^N  \theta_k \left(\sum_{{\substack{\Delta X^\circ_s \in B_k,\\ 0<s\leq T}}} \!\!\Delta X^\circ_s -  \sum_{\mathclap{\substack{\Delta X_s \in B_k,\\ 0<s\leq T}}} \Delta X_s \right) \\
& - \sum_{k=1}^N \rho_k(\mu^\circ_T(B_k) - \mu_T(B_k))
\end{split}
\end{equation}
for 
$X^{\circ} = (X_t^{\circ}\colon t \in [0, T])$ and 
$X = (X_t\colon t \in [0, T])$ with $X_T = X^\circ_T$, where $\mu^\circ_T(B_k)$ is defined analogously to $\mu_T(B_k)$ using $X^\circ$ instead $X$. Note that in this situation the righthand side is independent of the choice of the $\alpha$ parameter of the Gamma process measure used as the dominating measure.

\section{Sampling the posterior}
\label{sec:posterior}

Using the usual convention in Bayesian statistics, denote the prior density of the parameters
$\vartheta = (\alpha, \theta_1, \rho_1, \dots, \theta_N, \rho_N)$
by $\pi(\vartheta)$, and use a similar generic notation $q(\vartheta; \vartheta^\circ)$ for the density of the corresponding (joint) proposal kernel evaluated in $\vartheta^\circ$, e.g.~for a  random move from $\vartheta$ to $\vartheta^\circ$. 
We first describe the  Metropolis--Hastings algorithm to sample from the posterior in continuous time and next make a remark about the discretisation below.
\begin{itemize}
\item
Initialise the parameters $\alpha$, $\theta_k$, $\rho_k$, $k = 1$, $\dots$, $N$, with their  starting values. Initialise the segments
$(X_t\colon t_{i-1}\le t \le t_i)$ with $\Ga(\beta,\alpha)$ bridges connecting observations $X_{t_{i-1}}$ and $X_{t_{i}}$, $i = 1, \dots, n$, using \eqref{forward}.
\item
Repeat the following steps:
\begin{enumerate}
\item Independently, for each $i = 1, \dots, n$: \\ 
\begin{enumerate}
\item Sample $\Ga(\beta,\alpha)$ bridge proposals $(X^\circ_t\colon t_{i-1}\le t \le t_i)$ connecting observations $X_{t_{i-1}}$ and $X_{t_{i}}$ using 
\eqref{forward}.
 \item 
Sample $U_i \sim U[0,1]$.
If 
\begin{equation}\label{mh}
 \frac{\frac{\dd\pp_{\nu} }{\dd  \pp_{\nu_0}} (X^{\circ})}
 {\frac{\dd\pp_\nu}{\dd \pp_{\nu_0}} (X)}
 \ge U_i,
 \end{equation}
set $X_t$ to $X^\circ_t$ on $t_{i-1}\le t \le t_i$, otherwise keep $X_t$  on $t_{i-1}\le t \le t_i$. 
\end{enumerate}

\item Independently of step (i), propose $\vartheta^\circ \sim {q(\vartheta; \,\cdot\,)}$ and let $\nu^\circ$ denote the corresponding L\'evy measure.
Sample $U \sim U[0,1]$.
If 
\[
 \frac{\dd \pp_{\nu^\circ}}{\dd \pp_{\nu}}(X) 
 \frac{ \pi(\vartheta^\circ)  }{ \pi(\vartheta)  }
 \frac{q(\vartheta^\circ;\vartheta)}{q(\vartheta;\vartheta^\circ)}
  \ge U
\]
replace $\vartheta$ by  $\vartheta^\circ$, otherwise retain $\vartheta$. %

\end{enumerate}
\end{itemize}
Note that Step (i)(b) is the accept-reject step based on \eqref{formula:pathlr}.
Note that while we formulate the 

\subsection{Discretisation}\label{sec:discretisation1}

The Metropolis-Hastings algorithm described above assumes one can sample the various processes and their bridges in continuous time. In practice it is possible to simulate the relevant processes only on a discrete grid of time points, which, however, can be made arbitrarily fine. 
In general it is preferable to work with a finite-dimensional approximation of a valid MCMC algorithm with infinite-dimensional state space
instead of just an MCMC algorithm targeting a finite-dimensional approximation of the (joint) posterior, because the latter approach might have a singularity (resulting e.g.~in vanishing acceptance probabilities) with growing dimension; see \cite{beskos08} for an extended perspective. We now outline how our original algorithm can be discretised.
Consider a discrete time grid  $t_{i,j} = t_{i-1} + \frac{j}{m} (t_{i}-t_{i-1}) $ (and $t_n$) for $i = 1, \dots,  n$, $j = 0, \dots, m-1$.  Formula \eqref{forward} remains valid also for discretised Gamma processes, and those are readily obtained by sampling from the distribution of their increments.
On the other hand, in the likelihood expressions of Subsection \ref{subsec:likelihood:cases} and in \eqref{mh} we approximate the sum of jumps of the process $X$ with sizes in $B_k$, $k \geq 0$, by the sum of the increments of $X$ falling in $B_k$,
\begin{equation}
\label{eq:approx}
\sum_{\mathclap{\substack{\Delta X_s \in B_k,\\ 0<s\leq T}}} \Delta X_s \approx
\sum_i \sum_{j} (X_{t_{i,j}} - X_{t_{i,j-1}}) \one_{[X_{t_{i,j}} - X_{t_{i,j-1}} \in B_k]}.
\end{equation}

\section{Posterior consistency}\label{section:consistency}

In this section we study asymptotic frequentist properties of our nonparametric Bayesian procedure. The only comparable works  for L\'evy processes available in the literature are \cite{cpp15}, \cite{cpp16} and \cite{nickl17}, but they deal with the class of compound Poisson processes, which is quite different from the class of $\theta$-subordinators considered in this work. Arguments in favour of studying frequentist asymptotics for Bayesian procedures have been  already given in the literature many times, and will not be repeated here; see, e.g., \cite{wasserman98}. Our main result in this section is that under suitable regularity conditions, with growing sample size, our nonparametric Bayesian approach consistently recovers the parameters of interest. Thereby it stands on a solid theoretical ground.

\subsection{Main results}

Recall the setup of Section~\ref{sec:likelihood}, which is complemented as follows. In this section we assume that the process $X$ is observed at equidistant times $t_i$, $i=1,\ldots,n$. Without loss of generality we assume that our observations are $X_1,\ldots,X_n.$ 
{This assumption, which we did not require in earlier sections, implies that the increments of the process are independent and  identically distributed. This way we can develop our arguments without the additional technical burden caused by non-i.i.d.~increments. 
We denote the increments by  $\mathcal{Z}_n=\{Z_1,\ldots,Z_n\}$, where  $Z_i=X_i-X_{i-1},$ $i=1,\ldots,n,$  and} assume that under the true L\'evy density $v_0,$ $Z_1 \sim \mathbb{Q}_{v_0}$. In general, $\mathbb{Q}_{v}$ will stand for the law of the increment $Z_1$ under the L\'evy density $v.$ Furthermore, we introduce the law  $ \mathbb{P}_{v_0}$ of $(X_t\colon t\in[0,1])$ under the true L\'evy density $v_0.$ The law of this path under the L\'evy density $v$ will be denoted by $\mathbb{P}_v.$ For our asymptotic results, we will let the number of bins $N$ depend on the sample size $n$, and write $N_n$ instead. The prior $\Pi_n$ below will be defined on a special class of L\'evy densities, $V_n$. These are the densities that on the bins $B_k=(b_{k-1},b_k]$, $k=1,\ldots,N$, $b_0=0$, $b_1=\underline{b}$, $b_N=\overline{b}$, have the form $v(x)=\frac{\beta_0}{x}\exp(-\alpha x -\theta_k(x))$, with $\theta_k(x)=\rho_k+\theta_kx$, with the special choice $\rho_0=\theta_0=0$ and $\beta_0=1$. So, with the above notation,
\[
V_n=\left\{v\colon v_{|B_k}(x)=\frac{\beta_0}{x}\exp(-\alpha x -\theta_k(x)), \, k=1,\ldots,N\right\}.
\]
{
Below we present our first condition, and we comment on it and give additional explanations after it, as well as a few further comments after Condition~\ref{cnd:prior}.
}

\begin{cnd}
\label{cnd:truth}
Let the function $\theta_0$  have a compact support on the interval $[\underline{b},\overline{b}]$ where the boundary points $0<\underline{b}<\overline{b}<\infty$ are known, $\|\theta_0\|_\infty<\bar\theta$, and suppose  $\theta_0$ is $\lambda$-H\"older continuous, $|\theta_0(x)-\theta_0(y)|\leq L|x-y|^{\lambda}$ ($\lambda\in (0,1]$, $L>0$). Suppose also that $\alpha_0\in [\underline{\alpha},\overline{\alpha}]$ with known boundary points $0<\underline{\alpha}<\overline{\alpha}<\infty$. Finally, assume that the parameter $\beta_0$ is known and, without loss of generality, equal to $1$.
\end{cnd}

The assumption of known $\beta$ requires some further comments. As we already remarked elsewhere, the parameter $\beta$ plays a role similar to the dispersion coefficient $\sigma$ of a stochastic differential equation driven by a Wiener process. Derivation of nonparametric Bayesian asymptotics for the latter class of processes (all of which is a recent work) historically proceeded from the assumption of a known $\sigma$ to the one where $\sigma$ is unknown and has to be estimated; see \cite{frank13}, \cite{gugu14} and \cite{nickl17b}. In that sense the fact that at this stage we assume $\beta$ is known does not appear unexpected or unnatural. This assumption assists in derivation of useful bounds on the Kullback-Leibler and Hellinger distances between marginals of $\theta$-subordinators under different L\'evy triplets, which in general is the key to establishing consistency properties of nonparametric Bayesian procedures. We achieve this by reducing some of the intractable computations for these marginals to calculations involving laws of continuously observed $\theta$-subordinators, for which we need precisely to assume that the parameter $\beta$ is known; otherwise the corresponding laws are singular, which would yield only trivial and useless bounds.

\begin{cnd}
\label{cnd:prior}
The coefficients $\theta_i,$ $i=1,\ldots N-1,$ are equipped with independent uniform priors on the known interval $[-\overline{\theta},\overline{\theta}]$, $\overline\theta>0$.  Likewise, the coefficients $\rho_i,$ $i=1,\ldots,N-1,$   are independent uniform on the interval $[-\overline{\theta},\overline{\theta}],$ whereas $\alpha$ is uniform on $[\underline{\alpha},\overline{\alpha}]$, $\overline\alpha>0$.We assume that all priors are independent. Implicitly, this defines a prior $\Pi_n$ on the class of L\'evy densities $V_n$, which are realisations from the prior.
\end{cnd}
\par
The assumption in Condition \ref{cnd:prior} that various priors are uniform can be relaxed to the assumption that they are supported on compacts and have densities bounded away from zero there. In fact,  other assumptions in Conditions \ref{cnd:truth} and \ref{cnd:prior} can be relaxed at the cost of extra technical arguments in the proofs, but we do not strive for full generality in this work: a clean, readable presentation of our results and conciseness in the proofs is our primary goal.
\par
Theorem~\ref{thm:consistency} is our first main result in this section. Said shortly, it implies that our Bayesian procedure is consistent in probability; this in turn implies the existence of consistent Bayesian point estimates, see, e.g., \cite{ghosal00}, pp.~506--507.
We use the notation $\Pi_n(\dd v\mid \mathcal{Z}_n)$ for the posterior measure. Also, $\mathbb{Q}_{v_0}^n$ denotes the law of the sample $\mathcal{Z}_n$ under the true L\'evy density $v_0$ and $\mathbb{Q}_{v_0}^{\infty}$  denotes the law of the infinite sample $Z_1,Z_2,\ldots$ under the true L\'evy density $v_0$.

\begin{thm}
\label{thm:consistency}
Assume that Conditions~\ref{cnd:truth} and~\ref{cnd:prior} hold and that $N_n\to\infty$ and $N_n/n\to 0$ as $n\to\infty$.
Let $d_{\mathcal{H}}$ be the Hellinger metric. Then, for any fixed $\epsilon,\varepsilon>0,$
\[
\mathbb{Q}_{v_0}^n \left( \Pi_n(v\colon d_{\mathcal{H}}(\mathbb{Q}_{v_0},\mathbb{Q}_v) > \epsilon \mid \mathcal{Z}_n) > \varepsilon \right) \rightarrow 0
\]
as $n\rightarrow\infty.$ 
\end{thm}

{\
Before proceeding further, we recall the definition of the Kullback-Leibler divergence $\mathcal{KL}$ and the discrepancy $\mathcal{V}$ for two probability measures $\pp\ll \QQ$:
\[
\mathcal{KL}(\pp,\QQ)= \int \log \left(\frac{ \dd \pp}{ \dd \QQ} \right) \dd \pp, \quad \mathcal{V}(\pp,\QQ)= \int \log^2 \left(\frac{ \dd \pp}{ \dd \QQ} \right) \dd \pp.
\]
}
Here $\log^2$ stands for the square of the natural logarithm.
\par
\emph{Proof of Theorem~\ref{thm:consistency}.}
The technical results needed in the proof are collected in Appendix~\ref{lemmata0}. Write $B(\epsilon)=\{ v\colon d_{\mathcal{H}}(\mathbb{Q}_{v_0},\mathbb{Q}_v) \leq \epsilon \}$ and note that
\[
\Pi_n( B(\epsilon)^c \mid \mathcal{Z}_n) = \frac{\int_{B(\epsilon)^c} \prod_{i=1}^n \frac{\dd\mathbb{Q}_{v}}{\dd\mathbb{Q}_{v_0}}(Z_i) \Pi_n(\dd v) }{\int \prod_{i=1}^n \frac{\dd\mathbb{Q}_{v}}{\dd\mathbb{Q}_{v_0}}(Z_i) \Pi_n(\dd v)}=\frac{\textrm{Num}_n}{\textrm{Den}_n}.
\]
We will treat the numerator and denominator separately. We start with the denominator. Define the set
\[
K(\delta)=\{ v\colon \mathcal{KL}(\mathbb{Q}_{v_0},\mathbb{Q}_{v}) \leq \delta , \mathcal{V}(\mathbb{Q}_{v_0},\mathbb{Q}_{v}) \leq \delta \},
\]
where $\delta>0$ is a fixed number. Let $\widetilde{\Pi}_n$ be a restriction of the prior $\Pi_n$ to the set $K(\delta)$ normalised to have the total mass $1.$ We can write
\[
\textrm{Den}_n \geq \Pi_n(K(\delta)) \int_{K(\delta)} \prod_{i=1}^n \frac{\dd\mathbb{Q}_{v}}{\dd\mathbb{Q}_{v_0}}(Z_i) \widetilde{\Pi}_n(\dd v).
\]
By a standard argument as in \cite{ghosal00}, p.~525, using Lemmas \ref{lem:denominator} and \ref{lem:prior}, on the sequence of  events
\[
A_n=\left\{\int_{K(\delta)} \prod_{i=1}^n \frac{\dd\mathbb{Q}_{v}}{\dd\mathbb{Q}_{v_0}}(Z_i) \widetilde{\Pi}_n(\dd v)\geq e^{-Cn\delta}\right\}
\]
 of $\mathbb{Q}_{v_0}^n$-probability tending to $1$ as $n\rightarrow\infty,$
\begin{equation}
\label{denn}
\frac{1}{\textrm{Den}_n} \lesssim (c\delta)^{-2N_n} e^{Cn\delta}  \lesssim e^{\overline{\delta} n},
\end{equation}
 for $\overline{\delta}=2C\delta$, where for two sequences $\{a_n\}$ and $\{b_n\}$ of positive real numbers the notation  $a_n\lesssim b_n$ indicates that there exists a constant $C>0$ that is independent of $n$ such that $a_n\leq C b_n$.
 We also used the fact that $N_n/n\rightarrow0.$ For future use remember that $\overline{\delta}$ can be made arbitrarily small by choosing $\delta$ small.  This finishes bounding the term $\textrm{Den}_n.$
Now we turn to $\textrm{Num}_n.$ By Lemma \ref{lem:maximal}, on the sequence of events 
\[
B_n=\left\{\sup_{v\in B(\epsilon)^c}  \prod_{i=1}^n \frac{\dd\mathbb{Q}_{v}}{\dd\mathbb{Q}_{v_0}}(Z_i) < \exp(-c_1n\epsilon^2)\right\}
\] 
of $\mathbb{Q}_{v_0}^n$-probability tending to $1$ as $n\rightarrow\infty,$ we have
\begin{equation}
\label{numn}
\operatorname{Num}_n  \leq \exp(-c_1n\epsilon^2).
\end{equation}
The statement of the theorem now follows by choosing $\delta$ small enough, so that $\overline{\delta} < c_1 \epsilon^2$. Indeed, for all big $n$ one has on $A_n\cap B_n$
by combining the bounds \eqref{denn} and \eqref{numn} that $\Pi_n(v\colon d_{\mathcal{H}}(\mathbb{Q}_{v_0},\mathbb{Q}_v) > \epsilon \mid \mathcal{Z}_n) \leq \varepsilon $. Hence,
\[
\mathbb{Q}_{v_0}^n \left( \Pi_n(v\colon d_{\mathcal{H}}(\mathbb{Q}_{v_0},\mathbb{Q}_v) > \epsilon \mid \mathcal{Z}_n) > \varepsilon \right)\leq \mathbb{Q}_{v_0}^n(A_n^c\cup B_n^c)\to 0, 
\]
which proves the theorem.
\endproof
\par
The theorem has the following corollary that we will use in the proof of Theorem \ref{thm:consistency2}: a fixed $\epsilon$ can be replaced with a sufficiently slowly decaying  $\epsilon_n.$

\begin{cor}
\label{cor:epsn}
For every fixed $\varepsilon>0,$ there exists a sequence $\epsilon_n\rightarrow 0$, possibly depending on $\varepsilon,$ such that
\[
\mathbb{Q}_{v_0}^n \left( \Pi_n(v\colon d_{\mathcal{H}}(\mathbb{Q}_{v_0},\mathbb{Q}_v) > \epsilon_n \mid \mathcal{Z}_n) > \varepsilon \right) \rightarrow 0
\]
as $n\rightarrow\infty.$
\end{cor}
\par
\begin{proof}
The result follows from Lemma $\langle 22\rangle$ on p.~181 in \cite{pollard02}.
\end{proof}
\par
The metric for $v$, in which posterior convergence occurs in Theorem \ref{thm:consistency}, is defined indirectly, in terms of the distance between the corresponding laws $\mathbb{Q}_v,\mathbb{Q}_{v_0}.$ However, we will show that the theorem implies posterior consistency also in another and perhaps more natural metric for $v$. Let $\rightsquigarrow$ denote weak convergence of finite Borel measures and $\delta_0$ be the Dirac measure at zero. 
The following proposition holds, as a consequence of Theorem 2 in \cite{gnedenko39}, see Appendix~\ref{lemmata0} for its proof. Note that in our setting the first component of the L\'evy triplet is completely determined by the L\'evy density, cf.~\eqref{cond.gamma}. 
\begin{prop}
\label{lem:gnedenko}
Define for L\'evy triplets $(\gamma_n,0,\nu_n)$, $(\gamma,0,\nu)$ finite Borel measures
\[
\widetilde{\nu}_n(\dd x) = \gamma_n \delta_0(\dd x) + (x^2 \wedge 1)\nu_{n}(\dd x), \quad \widetilde{\nu}(\dd x) = \gamma \delta_0(\dd x) + (x^2 \wedge 1)\nu(\dd x),
\]
where we assume $\nu_n$ and $\nu$ are on \((0,\infty),\) and \(\gamma_n=\int_0^1 x\nu_{n}(\dd x)\) and \(\gamma=\int_0^1 x\nu(\dd x)\) are finite.
Then $\mathbb{Q}_{v_n} \rightsquigarrow \mathbb{Q}_{v}$ if and only if
$
\widetilde{\nu}_n \rightsquigarrow \widetilde{\nu}.
$
\end{prop}
\par
The following is our second main theoretical result, in which the metric for posterior contraction is defined directly for the L\'evy density $v$ (equivalently, L\'evy measure $\nu$). As the L\'evy density uniquely determines the corresponding L\'evy measure, in the theorem below as well as in its proof we will somewhat abuse the notation by considering posterior probabilities of certain sets of L\'evy measures.

\begin{thm}
\label{thm:consistency2}
Let $d_{\mathcal{W}}$ be any distance that metrises weak convergence of finite (signed) Borel measures. Then, for any fixed $\epsilon,\varepsilon>0,$ 
\[
\mathbb{Q}_{v_0}^n \left( \Pi_n(\nu\colon d_{\mathcal{W}}(\widetilde{\nu}_0,\widetilde{\nu}) > \epsilon \mid \mathcal{Z}_n) > \varepsilon \right) \rightarrow 0
\]
 as $n\rightarrow\infty$. \end{thm}

Since the L\'evy measures we consider are infinite in any neighbourhood of zero, using some weight function to convert them into finite measures does not appear to be an unnatural idea, cf.\ \cite{comte11} for a similar approach.

\emph{Proof of Theorem~\ref{thm:consistency2}.}
Note that Hellinger consistency in Theorem \ref{thm:consistency} also holds when we replace $d_{\mathcal{H}}$ with $d_{\mathcal{W}}$ there, since Hellinger consistency implies consistency in any metric metrising weak convergence. The proof of the theorem is by contradiction. Assume that the statement of the theorem fails, so that there exist $\epsilon,\varepsilon,\delta>0,$ such that
\begin{equation}
\label{contr}
\mathbb{Q}_{v_0}^n \left( \Pi_n( \nu\colon d_{\mathcal{W}}(\widetilde{\nu}_0,\widetilde{\nu}) > \epsilon \mid \mathcal{Z}_n) > \varepsilon \right) \geq \delta
\end{equation}
along a subsequence of $n,$ again denoted by $n$ for economy of notation. On the other hand, by Theorem \ref{thm:consistency} and Corollary \ref{cor:epsn} we know that for any $\varepsilon^{\prime},\delta^{\prime}>0$ there exists a sequence $\epsilon_n\rightarrow 0,$ such that for all $n$ large enough,
\begin{equation}
\label{fact}
\mathbb{Q}_{v_0}^n \left( \Pi_n( v\colon d_{\mathcal{W}}(\mathbb{Q}_{v_0},\mathbb{Q}_v) \leq \epsilon_n \mid \mathcal{Z}_n) > 1-\varepsilon^{\prime} \right) \geq 1-\delta^{\prime}.
\end{equation}
Take $\delta^{\prime}=\delta /2.$ Then the elementary relation
\[
P(A \cap B) = P(A)+P(B)-P(A\cup B)\geq P(A)+P(B)-1
\]
together with \eqref{contr}--\eqref{fact} imply that the intersection of the events
\begin{align*}
A_n&=\left\{  \Pi_n(\nu\colon d_{\mathcal{W}}(\widetilde{\nu}_0,\widetilde{\nu}) > \epsilon \mid \mathcal{Z}_n) > \varepsilon \right\}, \\ B_n&=\left\{ \Pi_n(v\colon d_{\mathcal{W}}(\mathbb{Q}_{v_0},\mathbb{Q}_v) \leq \epsilon_n \mid \mathcal{Z}_n) > 1-\varepsilon^{\prime} \right\}
\end{align*}
for all $n$ large enough has $\mathbb{Q}_{v_0}^n$-probability at least $\delta/2.$ In formula,
\begin{equation}
\label{qab}
\mathbb{Q}_{v_0}^n(A_n \cap B_n) \geq \delta/2.
\end{equation}
Let now $\varepsilon^{\prime}=\varepsilon/2,$ and suppose $\omega\in A_n \cap B_n.$ Then by the same argument as above, for the realisation $\mathcal{Z}_n(\omega),$ the intersection of two sets
\[
A^{\prime}=\{ \nu\colon d_{\mathcal{W}}(\widetilde{\nu}_0,\widetilde{\nu}) > \epsilon \}, \quad B_n^{\prime}= \{v\colon d_{\mathcal{W}}(\mathbb{Q}_{v_0},\mathbb{Q}_v) \leq \epsilon_n \}
\]
must have posterior mass at least $\varepsilon/2,$ for all $n$ large enough. Note that by this fact it also holds that
\[
A_n \cap B_n = \left\{ \Pi_n(A^{\prime}\cap B_n^{\prime}\mid\mathcal{Z}_n) \one_{[A_n\cap B_n]} \geq \varepsilon/2 \right\}.
\]
for all $n$ large enough. However, by Proposition \ref{lem:gnedenko} the intersection $A^{\prime}\cap B_n^{\prime}$ is an empty set for $n\rightarrow\infty$, so that 
\[
\Pi_n(A^{\prime}\cap B_n^{\prime}\mid\mathcal{Z}_n) \one_{[A_n\cap B_n]} \rightarrow 0, \quad \textrm{$\mathbb{Q}_{v_0}^{\infty}$-a.s.}
\]
But then, as $n\rightarrow \infty,$
\[
\mathbb{Q}_{v_0}^n \left( A_n \cap B_n \right) = \mathbb{Q}_{v_0}^n \left(  \Pi_n(A^{\prime}\cap B_n^{\prime}\mid\mathcal{Z}_n) \one_{[A_n\cap B_n]} \geq \varepsilon/2 \right) \rightarrow 0.
\]
This contradicts \eqref{qab}. The proof is completed.
\endproof

\section{Example: Sum of two Gamma processes}
\label{sec:example2}

Insurance theory, operational loss models, or more generally risk processes furnish a natural field of application for subordinators. In particular, a risk model based on Gamma process was extensively studied from a probabilistic point of view in the widely cited work \cite{dufresne91}. 
On the other hand, a given risk process may itself be a result of conflation of several heterogeneous factors, for instance due to population heterogeneity. We may assume that individual risk processes can be modelled through independent Gamma processes. This is conceptually similar to using convolutions of gamma distributions in, e.g., storage models; see \cite{mathai82}. The cumulative risk process is again a L\'evy process, though not necessarily gamma, as sums of independent Gamma processes are not necessarily Gamma. However, such sums can be closely approximated through $\theta$-subordinators, as we will now demonstrate. It is enough to consider the particular case of a sum of two independent Gamma processes, the general case being only notationally more complex. Thus, let $\widetilde{X}=(\widetilde{X}_t\colon t\geq 0)$ and $\hat{X}=(\hat{X}_t\colon t\geq 0)$ be two independent Gamma processes with parameters $(\beta_1,\alpha_1)$ and $(\beta_2,\alpha_2).$ Let the process $X=(X_t\colon t\geq 0)$ be their sum, $X_t=\widetilde{X}_{t}+\hat{X}_t.$ Its L\'evy density is given by
\[
v(x)=\frac{\beta_1}{x}e^{-\alpha_1 x}+\frac{\beta_2}{x}e^{-\alpha_2 x}.
\]
The process $X$ can be viewed as a mixture of phenomena happening at different time scales (slow and fast). For $x\to \infty$, the behaviour of $v$ is determined by $\beta_1 + \beta_2$ and $\min(\alpha_1, \alpha_2)$. 
On the hand, consider the equation
\[
\frac{\beta_1}{x}e^{-\alpha_1 x}+\frac{\beta_2}{x}e^{-\alpha_2 x}=\frac{\beta_1+\beta_2}{x}e^{-\theta(x)-\alpha x},
\]
where $\alpha>0$ will be chosen later on. Solving for $\theta,$ we get
\begin{equation}\label{sg:truth}
\theta(x)=-\log\left( \frac{\beta_1 e^{-\alpha_1 x}+\beta_2 e^{-\alpha_2 x}}{\beta_1+\beta_2} \right)-\alpha x.
\end{equation}
Now note that for $x\rightarrow 0,$
\[
-\log\left( \frac{\beta_1 e^{-\alpha_1 x}+\beta_2 e^{-\alpha_2 x}}{\beta_1+\beta_2} \right) \approx \frac{\beta_1\alpha_1+\beta_2\alpha_2}{\beta_1+\beta_2}x.
\]
We then take
\[
\alpha=\frac{\beta_1\alpha_1+\beta_2\alpha_2}{\beta_1+\beta_2}.
\]
This choice of $\alpha$ implies that the function $\theta$ is negligibly small in a neighbourhood of zero ($\theta(x)$ behaves as $x^2$ for $x$ small). It then follows that the L\'evy density of a sum of two independent Gamma processes can be closely approximated  by the L\'evy measure of the type \eqref{levy}, where $\theta$ is piecewise linear as in \eqref{piecewise}. Thus, $\theta$-subordinators can be used to approximate, to an arbitrary degree of accuracy, sums of independent Gamma processes.
For an illustration, see Figure~\ref{sumgamma:v}, that plots the function $x\mapsto-\log (x v(x))$  together with the corresponding slope $\alpha$ at $x=0$, and the asymptote $\min(\alpha_1,\alpha_2) x -\operatorname{const}$ for Example \ref{sumgamma} below.

\begin{figure}[htbp]
\begin{center}
\includegraphics[width=\textwidth]{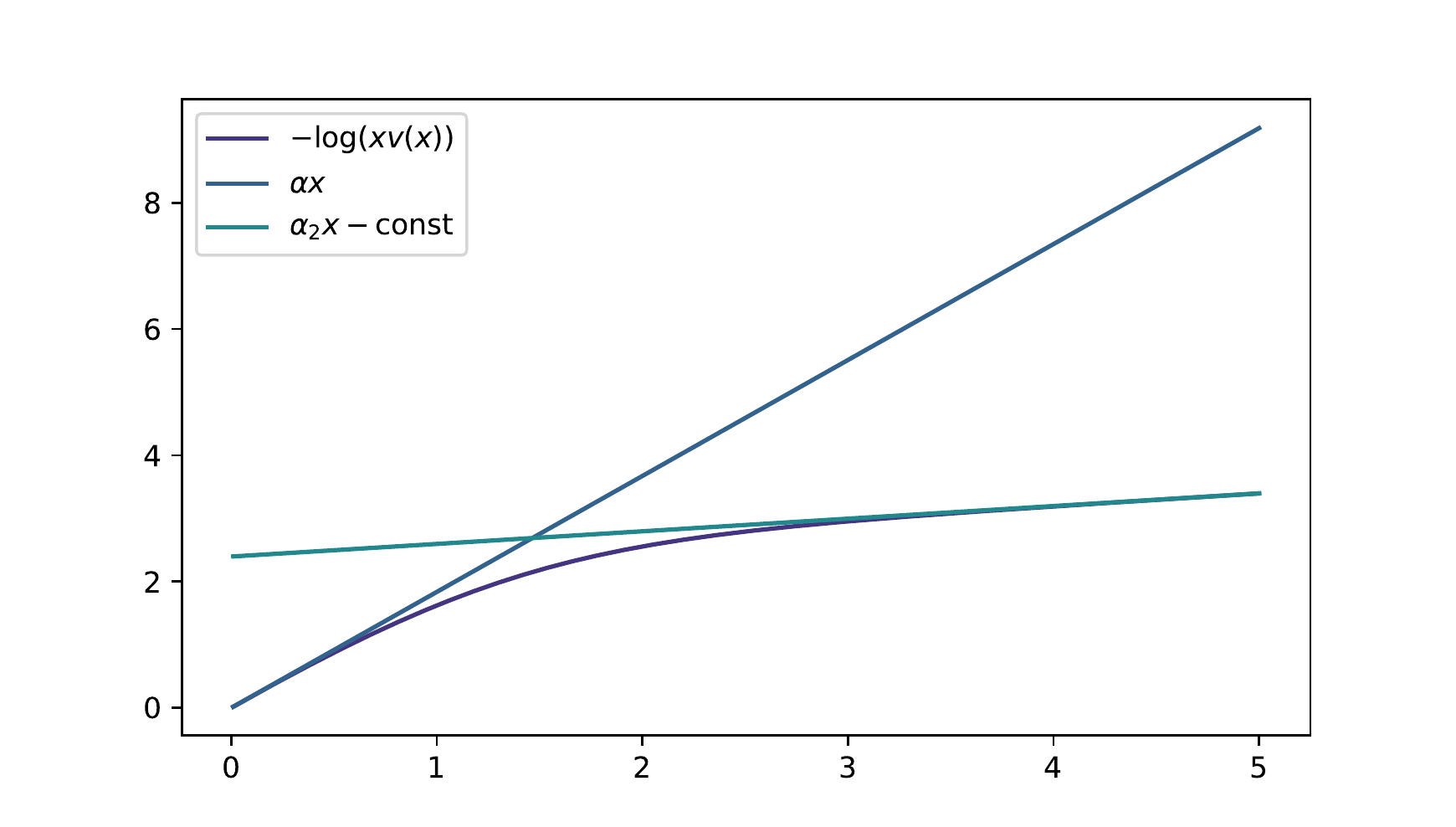}
\caption{
The function $x\mapsto-\log (x v(x))$ in Example \ref{sumgamma} together with the corresponding slope $\alpha$ at $x=0$, and the asymptote $\alpha_2x - \operatorname{const} = \min(\alpha_1,\alpha_2) x -\operatorname{const}$.
}
\label{sumgamma:v}
\end{center}
\end{figure}

We now consider a numerical example. All the computations in this work are performed using the software 
package \textbf{Bridge} (\cite{bridge17}) available for the Julia programming language, see \cite{bezanson17}. 

\begin{example}\label{sumgamma}
\textup{
{\
For the simulation of the synthetic data we chose 
$\alpha_1 = 2.0$, $\beta_1 = 0.4$,
$\alpha_2 = 0.2$, $\beta_2 = 0.04$.
For these parameters the behaviour sample paths of both components is neither too similar nor too far apart  (as judged by consulting Figure~\ref{sumgamma:v}),
making this an interesting statistical problem.
We simulated the process up to time 
$T = 2000$ and took $n = 10\,000$ observations at distance $0.2$.
}
}

\end{example}

For the prior we chose $N=3$ with grid points $b = [1, 2, 4]$, $\alpha \sim \operatorname{Gamma}(2, 1)$, $\theta_k \sim 
N(0,10)$ and $\rho_k \sim N(0, 50)$, $k \ge 1$, conditional on the realisation fulfilling $\theta(x) \to \infty$ for $x \to \infty$.

In the data augmentation step we took intermediate points at distance $0.01$. 

In the Gibbs sampler in each step new Gamma bridges are proposed in the data augmentation step, followed by a parameter update Metropolis-Hastings step with normal random walk proposals.
For the joint parameter update, using independent standard normal (Gaussian) innovations $Z_\alpha, Z_\theta, Z_\rho$ of appropriate dimensions,
we set
\begin{align*}
\alpha &= \alpha + \sigma_\alpha Z_\alpha,\\
\theta^\circ &= \theta + \sigma_\theta Z_\theta - (\alpha^\circ - \alpha),\\
\rho^\circ &= \rho + \sigma_\rho Z_\rho,
\end{align*}
where 
$\sigma_\alpha = \sigma_\theta = 0.025$, $\sigma_\rho = 0.15$.

The MCMC algorithm was run for $200\,000$ iterations. 
Figure \ref{sumgamma:traceplot}
shows trace plots and running averages of the posterior samples of the 
parameters $\alpha$ and $\theta_1$, $\theta_2$, $\theta_3$,  $\rho_1$, $\rho_2$, $\rho_3$. 
Figure \ref{sumgamma:bands} shows marginal Bayesian credible bands for the function $\theta(x)+\alpha x$ contrasted with the true function given by \eqref{sg:truth}. 
{As evidenced by the size of the marginal posterior bands, for bins chosen as indicated the observations do contain information about the L\'evy density on each bin.}

\begin{figure}[htbp]
\begin{center}
\includegraphics[width=0.375\textwidth]{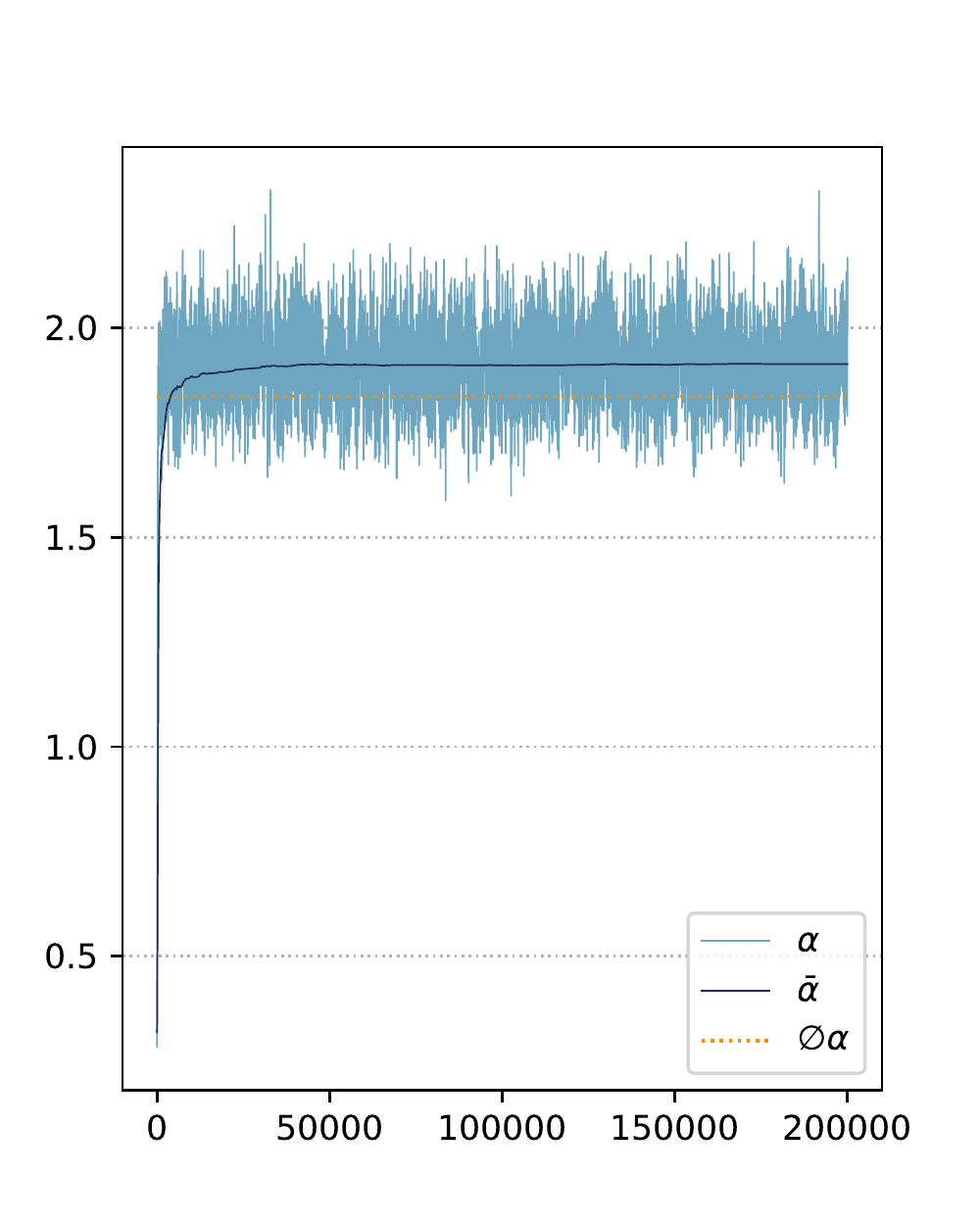}\includegraphics[width=0.625\textwidth]{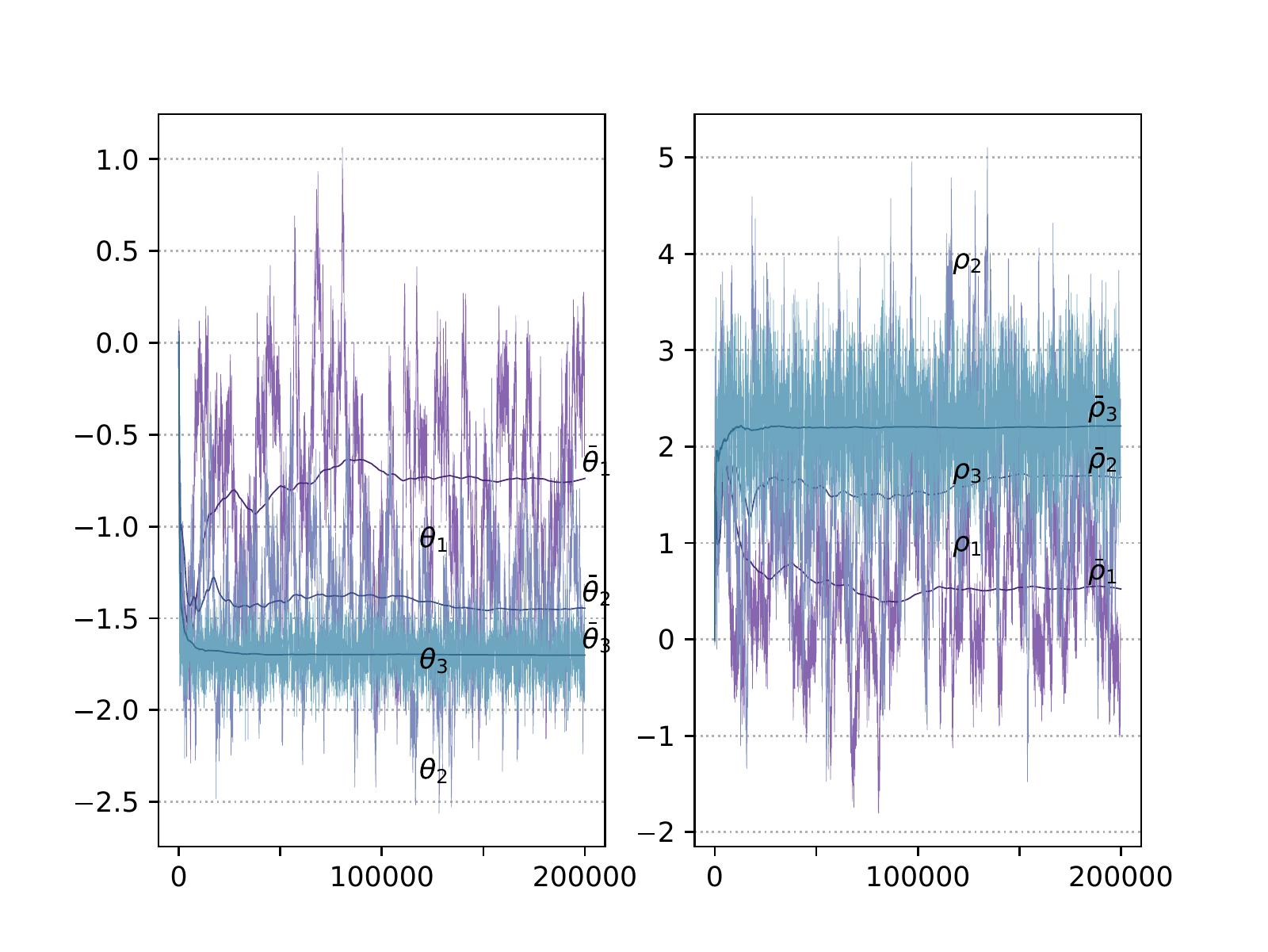}
\caption{Trace plots of the parameters for Example \ref{sumgamma}. First panel: trace and running average of samples of $\alpha$,
$\alpha=\frac{\beta_1\alpha_1+\beta_2\alpha_2}{\beta_1+\beta_2}$ is marked as yellow  line.
Second panel: trace and running average of samples of $\theta_1$, $\theta_2$, $\theta_3$.
Last panel: trace and running average of samples of $\rho_1$, $\rho_2$, $\rho_3$. Running averages of posterior samples of parameters are indicated by decorating the parameter with a bar.
}
\label{sumgamma:traceplot}
\end{center}
\end{figure}

\begin{figure}[htbp]
\begin{center}
\includegraphics[width=\textwidth]{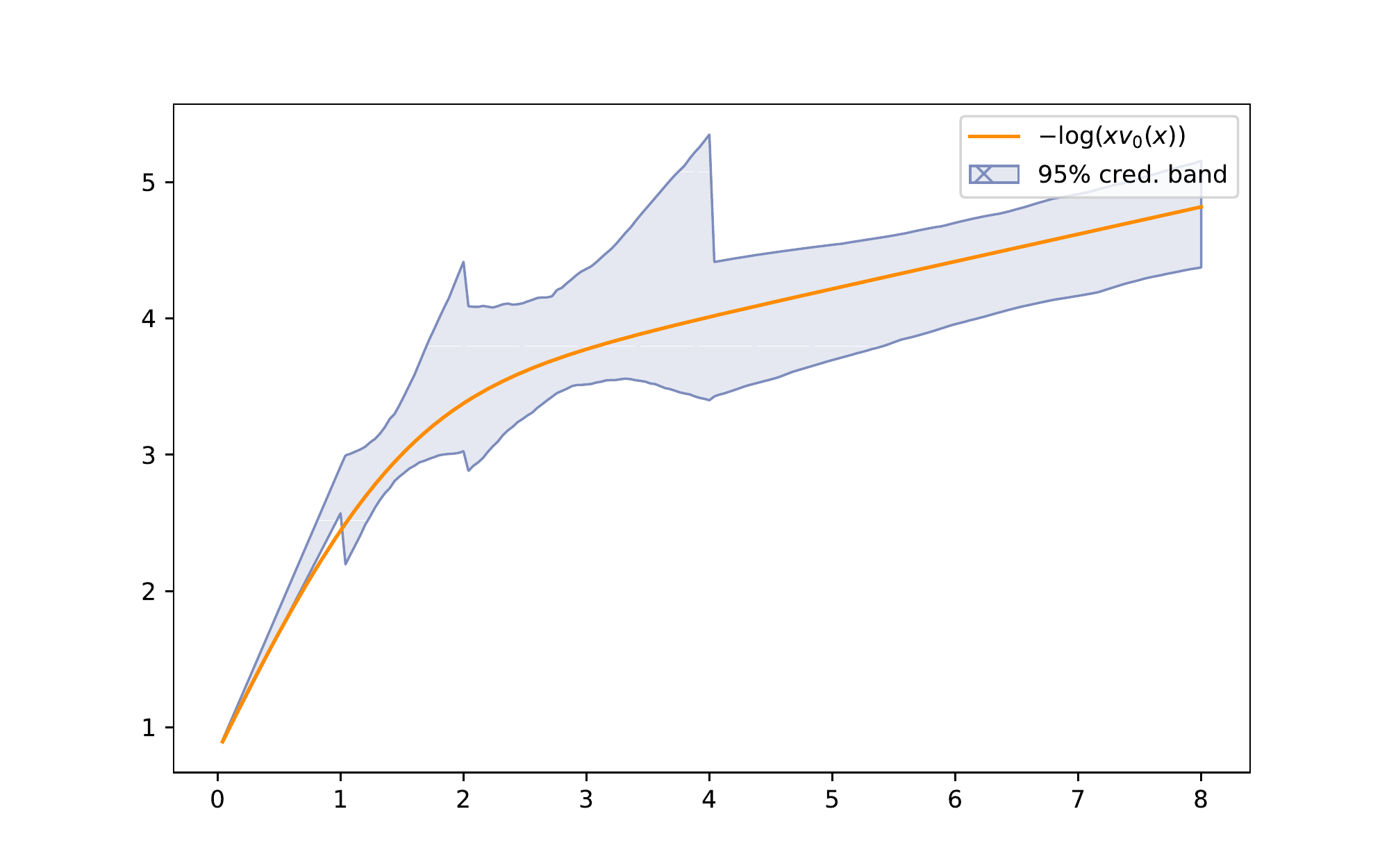}
\caption{Marginal Bayesian credible bands for Example \ref{sumgamma} for the function $\theta(x)+\alpha x$, based on all samples.
Orange: truth from equation \eqref{sg:truth}. 
}
\label{sumgamma:bands}
\end{center}
\end{figure}

\section{Estimation of $\beta$}

\label{sec:beta}

Thus far we assumed the parameter $\beta$ in \eqref{levy} is known. In practice such an assumption cannot always be justified, and the question arises how to adapt our Bayesian computational methodology to the case of an unknown $\beta.$ It should be noted that when viewed from a Bayesian data augmentation point of view, the parameter $\beta$ is rather different from the parameters $\alpha,\theta$: knowledge of $\beta$ is required in order to write down the likelihood of a  continuously observed process $X.$ As we noted before, in a sense, the parameter $\beta_0$ plays a role similar to the diffusion coefficient of the stochastic differential equation driven by the Wiener process.  Over the years, computational methods for handling the case of the unknown diffusion coefficient have been developed in the literature, see, e.g., \cite{schauer17} and references therein.
The basic idea of one such approach is that the laws of the bridge proposals can be understood
as push forwards of the laws of some underlying random processes. For Gamma process bridges (our bridge proposals) such a push forward map is given by \eqref{forward} and $\tilde\pp_\beta$ is the law of a Gamma process with parameter $\beta$. In the case of diffusion processes, where the bridge proposals are defined as strong solutions of stochastic differential equations, the law $\tilde \pp$  of the driving Brownian motion serves this purpose as a single law common to all models with different diffusion coefficients $\sigma^2$.  In our L\'evy setting the laws are different -- and mutually singular -- but are chosen in such a way that Metropolis-Hastings steps from one law $\tilde \pp_\beta$ to another $\tilde \pp_{\beta^\circ}$ can be balanced.
\par
We now move to providing details of our approach. Making use of the Markov property of a L\'evy process, we can restrict our attention to the case of a single bridge segment from
$0$ at time $t=0$ to $x_T$ at time $t=T.$ A generalisation to several bridges is straightforward. Since in our MCMC sampler for the posterior in an update step for the parameter $\beta,$ we will keep all other parameters fixed, in this section we can assume all the parameters except $\beta$ are known and fixed.
In what follows, $\pp_{\beta}$ denotes the law of a L\'evy process with L\'evy measure
\[
\nu(\dd x) = \frac{\beta}{x} {\rm e}^{-\alpha x - \theta(x)},
\]
and
$\tilde\pp_{\beta}$ denotes the law of a $\Ga(\beta,\alpha)$ process $\tilde X$, both defined on $[0,T]$. 
Next, $p_\beta$ and $\tilde p_\beta$ denote marginal densities of $X_T$ and $\tilde X_T$; furthermore, conditional laws (under $\pp_{\beta}$ and $\tilde\pp_{\beta}$) of the full L\'evy path given $X_T=x_T$ are denoted
$\pp^\star_\beta$ and  $\tilde \pp^\star_\beta$.
The map $g$ defined in \eqref{forward}  is written as $g_{x_T} = g_{0, x_T}$. Table \ref{table:notation} summarises the notation for easy reference.

\begin{table}[htp]
\begin{center}
\begin{tabular}{@{}lllll@{}}
\toprule
Process &  Law & Marginal density at $t=T$ & Law conditional on $X_T = x_T$ \\
 \hline
$X$& $\pp_\beta$  & $p_\beta$ &$\pp^\star_\beta$ \\
$\tilde X$ &  $\tilde\pp_{\beta}$ & $\tilde p_\beta \sim \Ga(t\beta,\alpha)$ & $\tilde \pp^\star_\beta$\\
$\tilde X^\circ$ & $\QQ_{\beta, \beta^\circ}( \tilde X;\, \cdot \,)$ &-- &--\\
\bottomrule
\end{tabular}
\end{center}
\caption{Notation chart for Section \ref{sec:beta}.}
\label{table:notation}
\end{table}%

Let $\beta$ be equipped with a prior distribution $\Pi$ assumed to be given by a density $\pi$. With   $\Psi_\beta = \frac{\dd \pp_\beta}{\dd \tilde \pp_\beta} $, the joint posterior of $(\beta, X)$ given $X_T = x_T$ can be factorised  as
\begin{equation}
\label{formula:post}
\begin{split}
\Pi((\dd \beta, \dd X) \mid x_T) & \propto \pi(\beta) p_\beta(x_T) \frac{\dd \pp^\star_\beta}{\dd \tilde \pp^\star_\beta} \bigl(X\bigr)  \tilde \pp^\star_\beta({\dd} X) \dd \beta \\
&= \pi(\beta)\tilde  p_\beta(x_T)\Psi_\beta \bigl(X\bigr) \tilde \pp^\star_\beta(\dd X) \dd \beta,
\end{split}
\end{equation}
where the second equality follows from \eqref{formula:bayes}.

Define a measure 
\begin{equation}\label{Lambdalaw}
\Lambda(\dd \beta, \dd \tilde X) = \pi(\beta)\tilde  p_\beta(x_T)\Psi_\beta \bigl(g_{x_T}(\tilde X)\bigr) \tilde \pp_\beta(\dd \tilde X) \dd \beta .
\end{equation}
Then $\Pi((\dd \beta, \dd X) \mid x_T)$ is proportional to the image measure of $\Lambda(\dd \beta, \dd \tilde X)$ under $(\beta, \tilde X) \mapsto (\beta, g_{x_T}(\tilde X))$,
because $g_{x_T}(\tilde X)\sim  \tilde \pp^\star_\beta$ for $\tilde X \sim  \tilde \pp_\beta$.  Note that $\Lambda$ does not involve the intractable density $p_\beta$, and 
$\Psi_\beta$ is analytically known, cf.~\eqref{formula:pathlr}.

We define a Metropolis-Hastings chain with $\Lambda$ as its invariant measure, from which samples $(\beta, g_{x_T}(\tilde X))$ of the joint posterior in \eqref{formula:post} are obtained.  As $g_{x_T}$ is not invertible, this is a data augmentation procedure, only that $\tilde X$, unlike the augmented path, can hardly be interpreted as an unobserved object.

Let $\tilde X$ be a $\Ga(\beta, \alpha)$ process
and assume that a proposal density for $\beta^\circ$ is given by $q(\beta; \beta^\circ)$. For a given $\beta^\circ$, if $\beta^\circ > \beta$, set
$
\tilde X^\circ_t = \tilde X_t + \tilde X'_t,
$ 
where $\tilde X' \sim \tilde\pp'$  is an independent $\Ga(\beta^\circ - \beta,\alpha)$ process. 
If  $\beta^\circ < \beta$, then set
\[
 \tilde X^\circ_t = \sum_{\mathclap{\substack{ \Delta \tilde X_s > 0\\ s \le t}}} U_s\Delta \tilde X_s,
 \]
where $U_s$ is an independent collection of $\operatorname{Bernoulli}( \beta^\circ/\beta)$ random variables indexed by a countable set $\{s\colon \Delta \tilde X_s > 0\}$.
By Lemma \ref{lem:reversible}~(i) and (ii) ahead, $\tilde X^\circ$ is a  $\Ga(\beta^\circ ,\alpha)$ process with law
$\tilde\pp_{\beta^\circ}$. Denote the probability kernel for a transition from $\tilde X$ to $\tilde X^\circ$ (conditional on $\beta$ and $\beta^\circ$), which is implied by the preceding construction, by $\QQ_{\beta, \beta^\circ}(\tilde X; \,\cdot\,)$.

We will show that proposing a move from $\beta$ to $\beta^\circ$ from $q$ and subsequently from $\tilde X$ to  $\tilde X^\circ $ and accepting it with acceptance probability 
$
  A((\beta, \tilde X), (\beta^\circ, \tilde X^\circ)) 
$
to be derived below, is a reversible move for $\Lambda$. By \cite{tierney98}, this follows if detailed balance
\begin{multline*}
\Lambda(\dd \beta, \dd \tilde X ) q(\beta; \beta^\circ)   \QQ_{\beta, \beta^\circ}((\beta, \tilde X); (\dd \beta^\circ, \dd \tilde X^\circ))  A((\beta, \tilde X), (\beta^\circ, \tilde X^\circ))  \dd \beta^\circ\\
= \Lambda(\dd \beta^\circ, \dd \tilde X^\circ) q(\beta^\circ; \beta)   \QQ_{\beta^\circ, \beta} ((\beta^\circ, \tilde X^\circ);(\dd \beta, \dd \tilde X))  A((\beta^\circ, \tilde X^\circ), (\beta, \tilde X))  \dd \beta
\end{multline*}
holds. By \eqref{Lambdalaw} and Lemma \ref{lem:reversible2} given below, the lefthand side is equal to
 \[
  \pi(\beta)\tilde  p_\beta(x_T)\Psi_\beta \bigl(g_{x_T}(\tilde X)\bigr) q(\beta; \beta^\circ)   \mu((\dd \beta, \dd \tilde X), (\dd \beta^\circ, \dd \tilde X^\circ))  A((\beta, \tilde X), (\beta^\circ, \tilde X^\circ))  
 \]
 with $\mu$ defined in  Lemma \ref{lem:reversible2} ahead.
 Therefore, choosing
\[
  A((\beta, \tilde X), (\beta^\circ, \tilde X^\circ)) = \max\left( \frac{\pi(\beta^\circ)}{\pi(\beta)} \frac{\tilde p_{\beta^\circ}(x_T) }{ \tilde p_\beta(x_T)} \frac{\Psi_{\beta^\circ}(g_{x_T}(\tilde X^\circ))}{\Psi_\beta(g_{x_T}(\tilde X))} 
\frac{q( \beta^\circ; \beta)}{q(\beta; \beta^\circ)}, \scalebox{1.2}{$1$} \right)
\]
can be seen to make the expressions on both sides of the last display equal, thanks to \eqref{Lambdalaw} and Lemma \ref{lem:reversible2} together with the symmetry of $\mu$ established in Lemma \ref{lem:reversible2}.

\begin{lemma}\label{lem:reversible}
Let $\tilde X_t = \sum_{{s \le t \colon \Delta \tilde X_s > 0}} \Delta \tilde X_s$ be a $\Ga(\beta, \alpha)$ process.
\begin{enumerate}
\item If $\beta^\circ > \beta$ and $X'$  is an independent $\Ga(\beta^\circ - \beta,\alpha)$ process, then 
\[
\tilde X^\circ_t = \tilde X_t + X_t^{\prime},
\] is 
a  $\Ga(\beta^\circ ,\alpha)$ process.
\item If  $\beta^\circ < \beta$ and $U_s$ is a countable collection of $\operatorname{Bernoulli}( \beta^\circ/\beta)$ random variables indexed by $\{s\colon \Delta \tilde X_s > 0\}$, then 
\[
\tilde X^\circ_t = \sum_{\mathclap{\substack{ \Delta \tilde X_s > 0\\ s \le t}}} U_s\Delta \tilde X_s 
 \]
is a  $\Ga(\beta^\circ ,\alpha)$ process.
\end{enumerate}
\end{lemma}
\begin{proof} We sketch the proof.
The first part is straightforward. The second part is more involved, but is a standard technique to sample L\'evy processes by thinning marked Poisson point processes, see the rejection method in \cite{rosinski01}; it could also be derived from the proof of Lemma \ref{lem:reversible2}. 
\end{proof}

\begin{lemma}[Transdimensional balance]\label{lem:reversible2}
For  $\beta, \beta^\circ > 0$,
\begin{equation}\label{transdim-balance}
\tilde\pp_{\beta}(\dd \tilde X)  \QQ_{\beta,  \beta^\circ}(\tilde X; \dd\tilde X^\circ) = \tilde\pp_{\beta^\circ}(\dd \tilde X^\circ)  \QQ_{\beta^\circ, \beta}( \tilde X^\circ; \dd \tilde X) 
\end{equation}
holds, and
\[
\mu((\dd \beta, \dd \tilde X), (\dd \beta^\circ, \dd \tilde X^\circ))  =  \dd \beta\dd \beta^\circ \tilde\pp_{\beta}(\dd \tilde X) \QQ_{\beta,  \beta^\circ}(\tilde X; \dd\tilde X^\circ) (= \mu((\dd \beta^\circ, \dd \tilde X^\circ),(\dd \beta, \dd \tilde X)) )
\]
defines a symmetric measure.
\end{lemma}

\begin{proof}
Without loss of generality, assume $\beta^\circ > \beta$.
The process $\tilde X$ is determined by the jump times $J^i = \{s \colon \Delta \tilde X_s \in [u_i,v_i)\}$ and jump sizes $
\Delta \tilde X_{s}$,  $s\in J^i$ on all disjoint strips $[0,T]\times[u_i, v_i)$, where 
$(0, \infty) = \bigcup_{i=1}^\infty [u_i, v_i) $ with $v_0 = \infty$, $v_i = 1/i$, $u_i = 1/(i+1)$. Similar to $J^i$, denote by $J^{i,\circ}$ the jump times of $\tilde X^\circ$ with their sizes in $[u_i, v_i)$.
The number of jumps $|J^i|$ is $\operatorname{Poisson}( \beta  c^i)$  distributed, with density written as $p^i_{\beta}(|J^i|)$), where
\[
c^i = T\tilde\nu([u_i,v_i))/\beta  = T\tilde\nu^\circ([u_i,v_i))/\beta^\circ.
\]

Conditional on $|J^i|$, the elements of $J^i$ are independent uniforms on $[0,T],$ and $
\Delta \tilde X_{s}$,  $s\in J^i$, are independently 
\begin{equation}\label{pdens} 
T\tilde\nu(\cdot)|_{[u_i, v_i)}/(\beta c^i) =T \tilde\nu^\circ(\cdot)|_{[u_i, v_i)}/(\beta^\circ c^i)
\end{equation} distributed; note that either side of \eqref{pdens} does not depend on $\beta$, which cancels from the formulae. 
Let $q^i_{\beta,\beta^{\circ}}(n; n^\circ)$ denote the counting density of moving from $|J^i| = n$  to $ |J^{i,\circ,}| = n^\circ$  under $\QQ_{\beta, \beta^\circ}(\tilde X;\, \cdot\,)$. This is well defined, as  $|J^{i,\circ,}|$ under  $\QQ_{\beta, \beta^\circ}(\tilde X; \, \cdot \,)$ only depends on $\tilde X$ through $|J^i|$.

On each strip it holds that
\begin{align*}
 p_\beta^i(|J^i|) & q_{\beta; \beta^\circ}^i(|J^i| ; |J^{i,\circ}|) 
= \frac{( \beta c^i)^ {|J^i|} e^{-\beta c^i}}{ {|J^i|}!}\frac{( (\beta^\circ - \beta) c^i)^{ |J^{i,\circ}| -  {|J^i|}} e^{-(\beta^\circ-\beta) c^i}}{( |J^{i,\circ}|- {|J^i|})!}\\
&= \frac{( \beta^\circ c^i)^{ |J^{i,\circ}|} e^{-\beta^\circ c^i}}{ |J^{i,\circ}|!} \binom{ |J^{i,\circ}|}{ {|J^i|}} \left(\frac\beta{\beta^\circ}\right)^ {|J^i|} \left(1-\frac\beta{\beta^\circ}\right)^{ {|J^i|}- |J^{i,\circ}|}
\\ &=  p^i_{\beta^\circ}( |J^{i,\circ}|)q_{\beta^\circ; \beta}^i( |J^{i,\circ}|;  {|J^i|}),
\end{align*}
as the number of jumps of $\tilde X'$ (the notation is as in Lemma \ref{lem:reversible}~(i)) in $[u_i, v_i)$ has the $\operatorname{Poisson}(  (\beta^\circ - \beta)  c^i)$ distribution. 
Note that 
\[
\prod_{s \in J^{i,\circ}} p((t^\circ_s, \Delta \tilde X^\circ_{s})) = \prod_{s \in J^i} p((t_s, \Delta \tilde X_{s})) \prod_{s \in J^{i,'}} p((t'_s, \Delta \tilde X'_{s}))
\]
where we used that the joint density $p$ is the same for all arguments by \eqref{pdens}.

Therefore on each strip it holds that
\begin{equation}
\label{loc:balance}
\tilde\pp_\beta(\dd \pi^i(\tilde X))  \QQ_{\beta,  \beta^\circ}(\pi^i(\tilde X);\dd \pi^i(\tilde X^\circ)) = \tilde\pp_{\beta^\circ}(\dd \pi^i(\tilde X^\circ))  \QQ_{\beta^\circ, \beta}( \pi^i(\tilde X^\circ); \dd \pi^i(\tilde X)),
\end{equation}
where $\pi^i\colon \tilde X\mapsto (|J^i|, \{(t_s,\Delta \tilde X_s), s \in J^i\})$. 
The statement of the lemma now follows from an application of Lemma \ref{lem:globalbalance}, by which \eqref{loc:balance} together with the independent increments property of the jump measure of a L\'evy process gives \eqref{transdim-balance} and thus also the symmetry of $\mu$.
\end{proof}

The terminology `transdimensional balance' for \eqref{transdim-balance} is suggested by a connection to the transdimensional MCMC in \cite{green1995}. In fact, note that for $\beta^\circ > \beta$, with $\tilde X \sim \tilde \pp_\beta$ and $\tilde X'$ as in Lemma \ref{lem:reversible}, the proposal
\[
 \tilde X^\circ_t = \begin{cases}
   \tilde X_{t \beta^\circ/\beta } &t \le \frac\beta{\beta^\circ}T\\
   \tilde X_{T} + \tilde X'_{\frac{\beta^\circ-\beta}{\beta^\circ} (t -\frac\beta{\beta^\circ}T)}  & t  > \frac\beta{\beta^\circ}T,
  \end{cases}
\]
 has also distribution $\tilde P_\beta$. This closely resembles the `standard template' given by \cite{green1995} for a transdimensional reversible jump move, although here all spaces are infinite-dimensional.

\subsection{Discretisation}\label{sec:discretisation2}
In order to be able to employ the result of this section in practice, we now discuss how to perform steps (i) and (ii) of Lemma \ref{lem:reversible} for the approximations defined on the discrete time grid as introduced in Subsection \ref{sec:discretisation1}.
Step (i) is straightforward, noting that for $\beta^\circ > \beta$,
\[
\tilde X^\circ_{t+h} - \tilde X^\circ_t \mid \tilde X_{t+h} -\tilde X_t\sim \tilde X_{t+h} -\tilde X_t + Z,
\]
where $Z \sim \Ga(h(\beta^\circ - \beta)\alpha)$.

For step (ii), when $\beta^\circ < \beta$, we use the following formula linking the law of the increments of the thinned process with the Beta distribution,
\[
\tilde X^\circ_{t+h} - \tilde X^\circ_t \mid \tilde X_{t+h} -\tilde X_t \sim \left(\tilde X_{t+h} -\tilde X_t\right)Z,
\]
where $Z \sim \operatorname{Beta}(h\beta^\circ,h (\beta-\beta^\circ))$.

\section{Example: sum of two Gamma processes, unknown $\beta$}
\label{sec:example:beta}

We revisit Example \ref{sumgamma}  from Section \ref{sec:example2}, but now additionally assuming the parameter $\beta$ is unknown.
We endow $\beta$ with an independent uniform prior on the interval $[0.1, 1000]$. To estimate $\beta$, we perform a transdimensional move, as explained in Section~\ref{sec:beta}, at every 5th iteration in the otherwise unchanged algorithm from Section \ref{sec:example2}.
Proposals for $\beta^\circ$ are obtained from a random walk with independent Gaussian increments, with standard deviation $\sigma_\beta = 0.01$. No further tuning is necessary.

Figure \ref{sumgamma:traceplot1b}
shows trace plots and running averages of the posterior samples of the 
parameters $\alpha$ and $\beta$. 
The data -- for the parameter values considered -- is informative for the parameter $\beta$ and the Metropolis-Hastings chain sampling from the posterior of $\beta$ mixes fast. While not covered by our posterior consistency result, the results of the numerical experiment indicate that the sampling procedure for $\beta$ integrates seamlessly into the algorithm. 
Figure \ref{sumgamma:traceplot2b}
shows trace plots and running averages of the posterior samples of $\theta_1$, $\theta_2$, $\theta_3$ and of $\rho_1$, $\rho_2$, $\rho_3$. 
Figure~\ref{sumgamma:hist1b} shows histograms of the posterior samples of $\alpha$ and $\beta$, whereas Figure~\ref{sumgamma:hist2b} shows histograms of the posterior samples of $\theta_1$, $\theta_2$, $\theta_3$ and of $\rho_1$, $\rho_2$, $\rho_3$.
Figure \ref{sumgamma:bandsb} shows marginal Bayesian 95\,\% credible bands for the function $-\log (x v(x))$ contrasted with the true function $-\log (x v_0(x))$ given by \eqref{sg:truth}.
The conclusion is that we are able to recover the qualitative properties (as indicated by the asymptotes in Figure~\ref{sumgamma:v}) of the process in both time scales from the discrete observations.

\begin{figure}[htbp]
\begin{center}
\includegraphics[width=\textwidth]{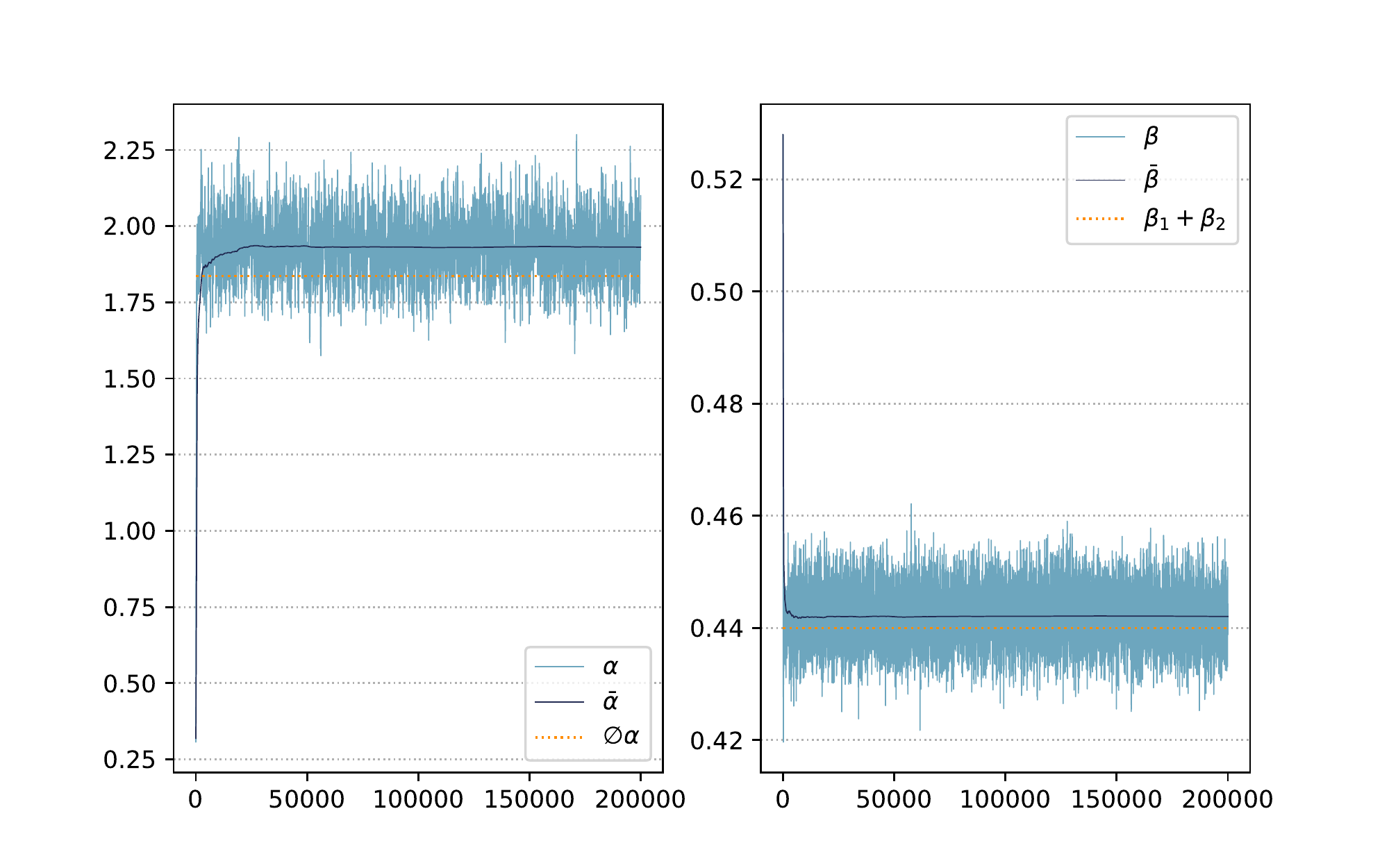}
\caption{Trace plots of the parameters $\alpha$ and $\beta$ for Example \ref{sumgamma}. Left: trace and running average ($\bar{\alpha}$) of samples of $\alpha$.
The value \o$\alpha=\frac{\beta_1\alpha_1+\beta_2\alpha_2}{\beta_1+\beta_2}$ is marked as a dotted yellow line.
Right: trace and running average of samples of $\beta$. The value $\beta_1 + \beta_2$ is marked as a dotted yellow line.
}
\label{sumgamma:traceplot1b}
\end{center}
\end{figure}

\begin{figure}[htbp]
\begin{center}
\includegraphics[width=\textwidth, height=3cm]{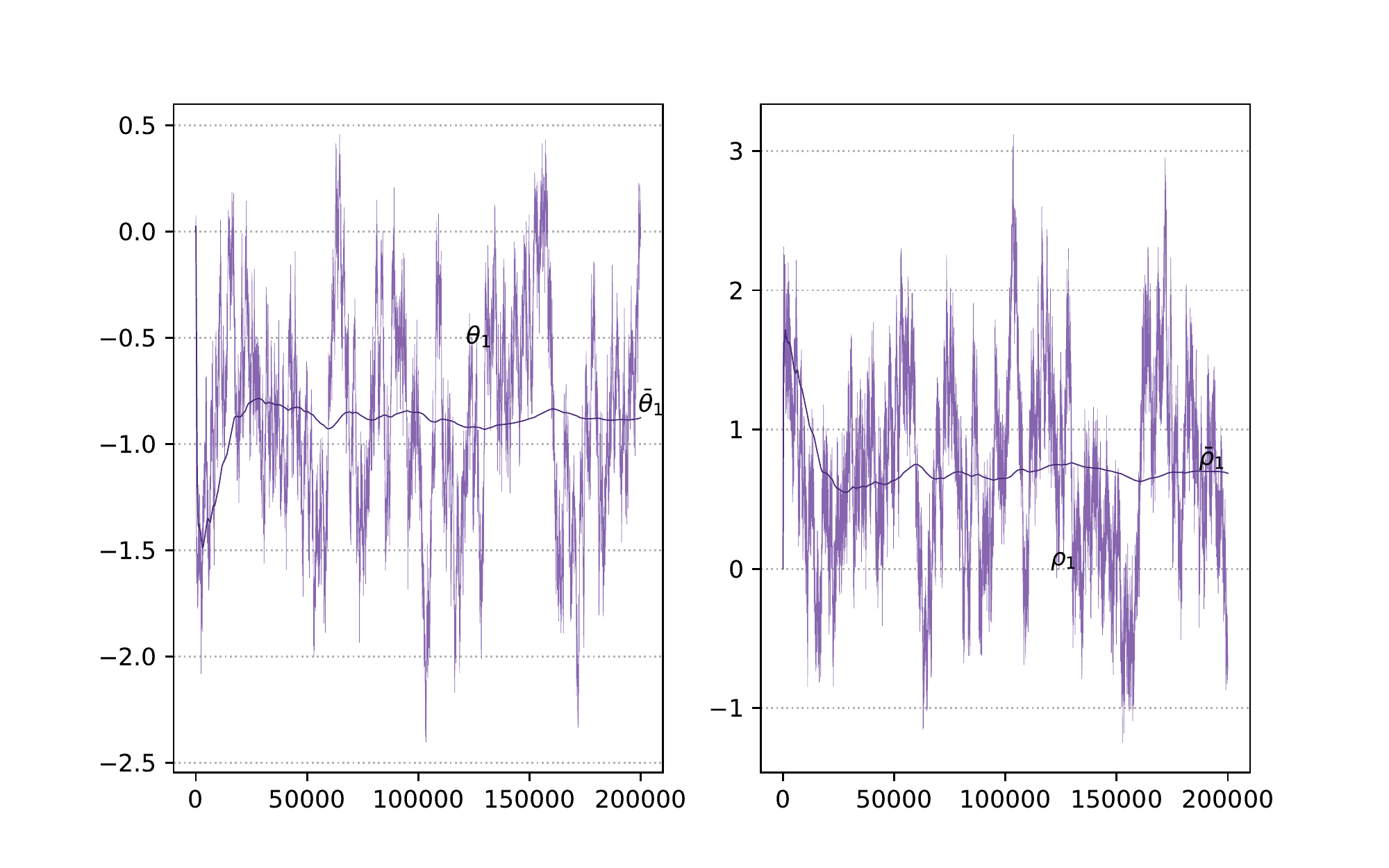}
\includegraphics[width=\textwidth, height=3cm]{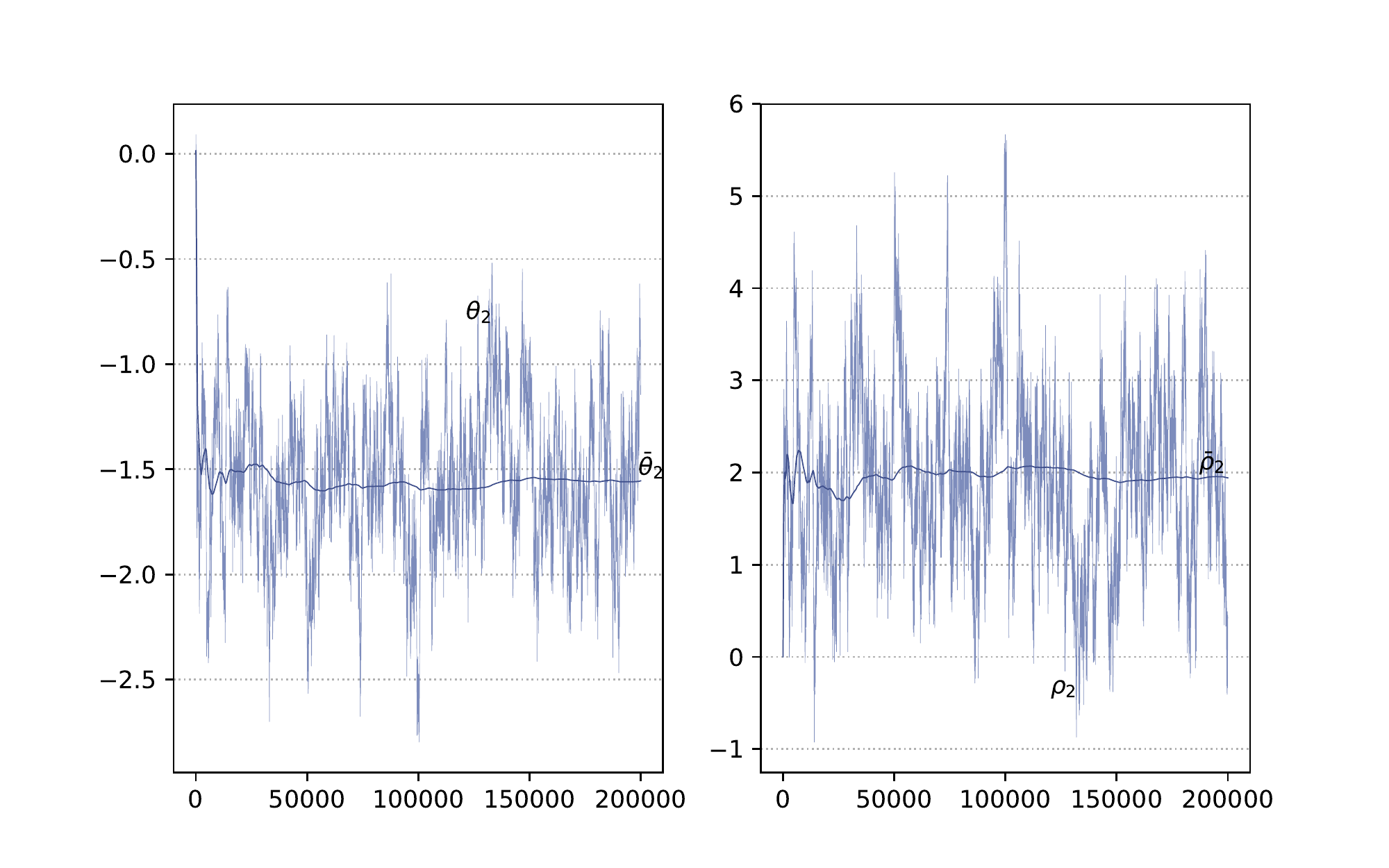}
\includegraphics[width=\textwidth, height=3cm]{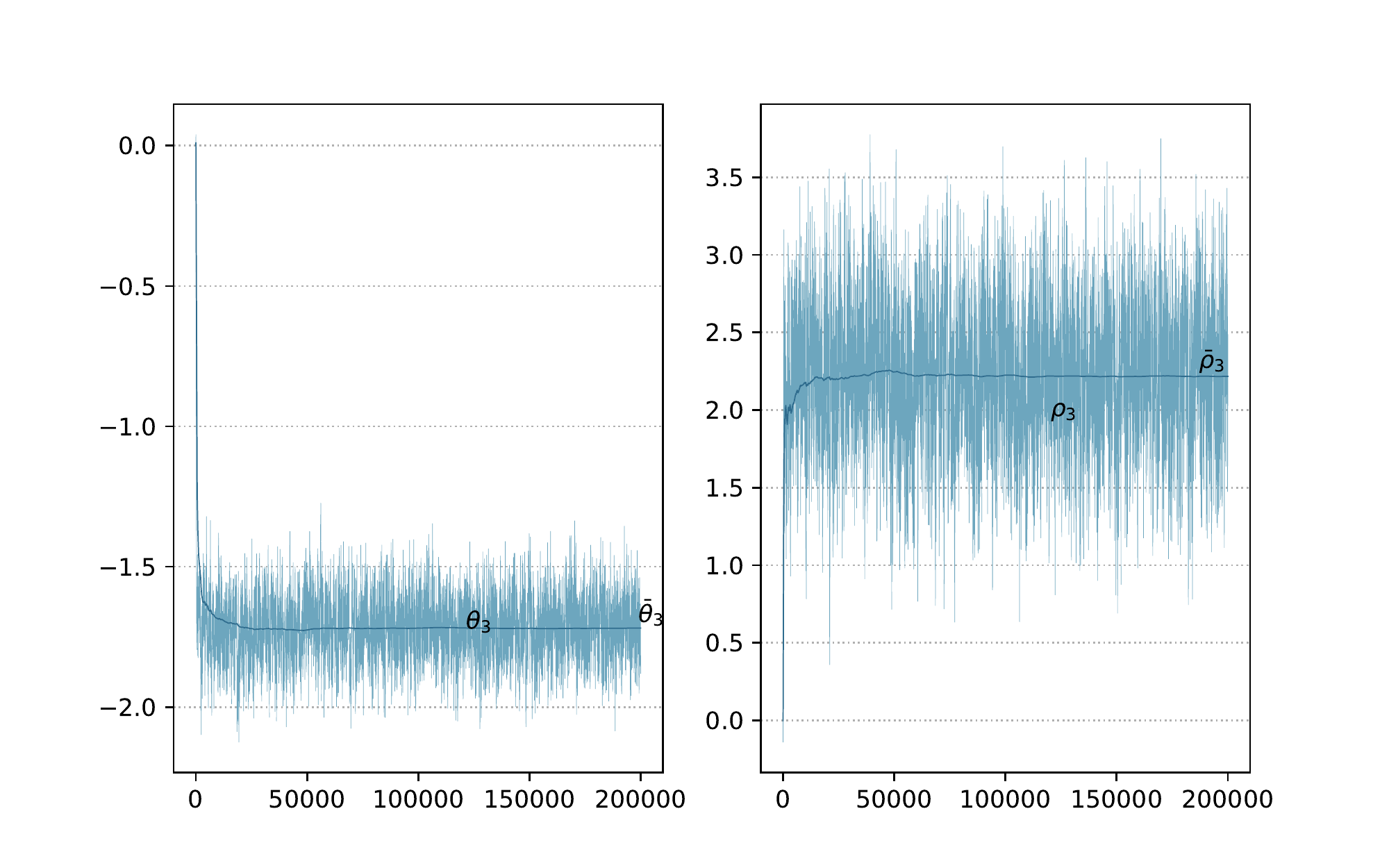}

\caption{Trace plots of the parameters for Example \ref{sumgamma}.  Left column: trace and running average of samples $\theta_1$, $\theta_2$, $\theta_3$.
Right column: trace and running average of samples of $\rho_1$, $\rho_2$, $\rho_3$, 
}
\label{sumgamma:traceplot2b}
\end{center}
\end{figure}

\begin{figure}[htbp]
\begin{center}
\includegraphics[width=0.5\textwidth]{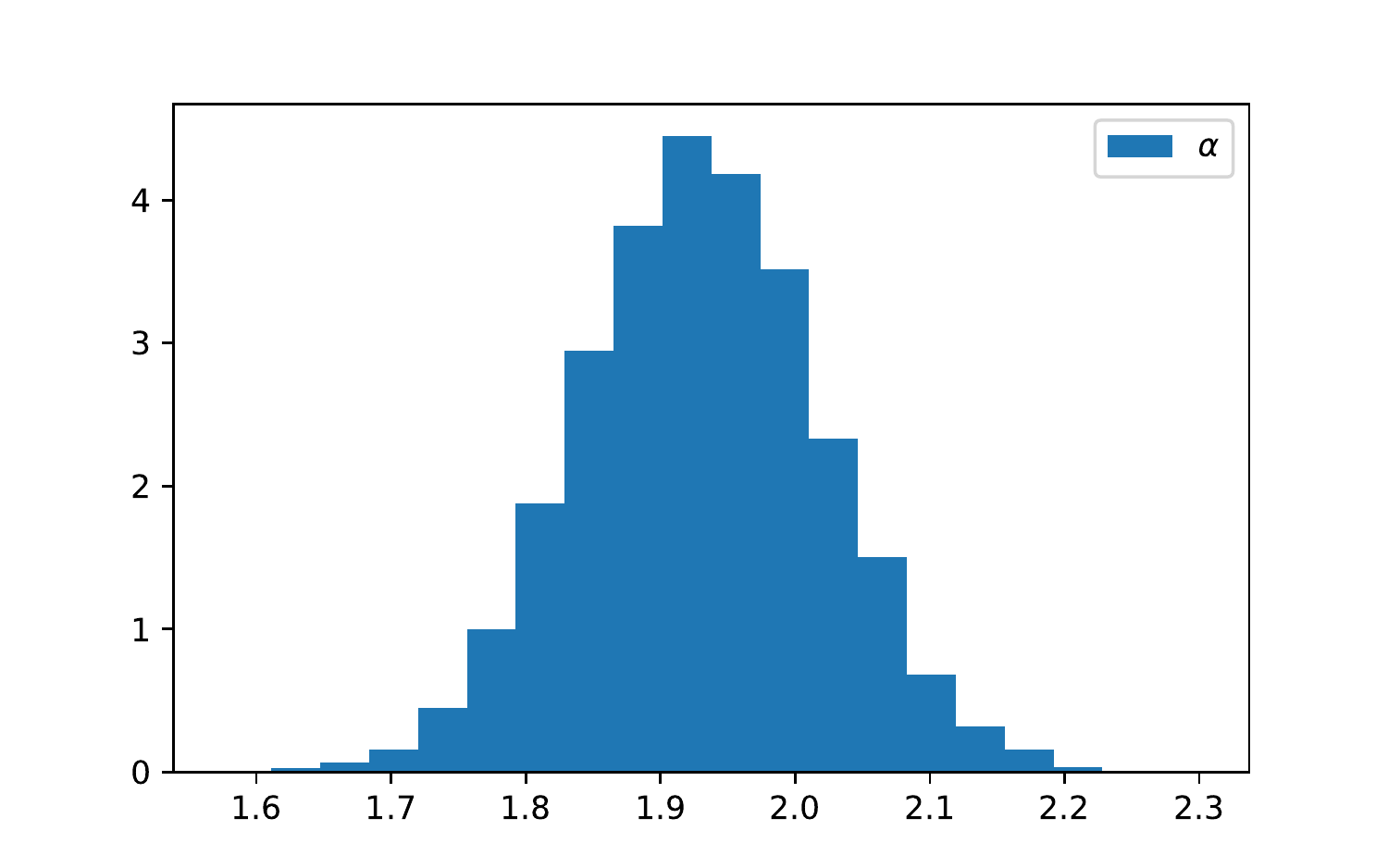}\includegraphics[width=0.5\textwidth]{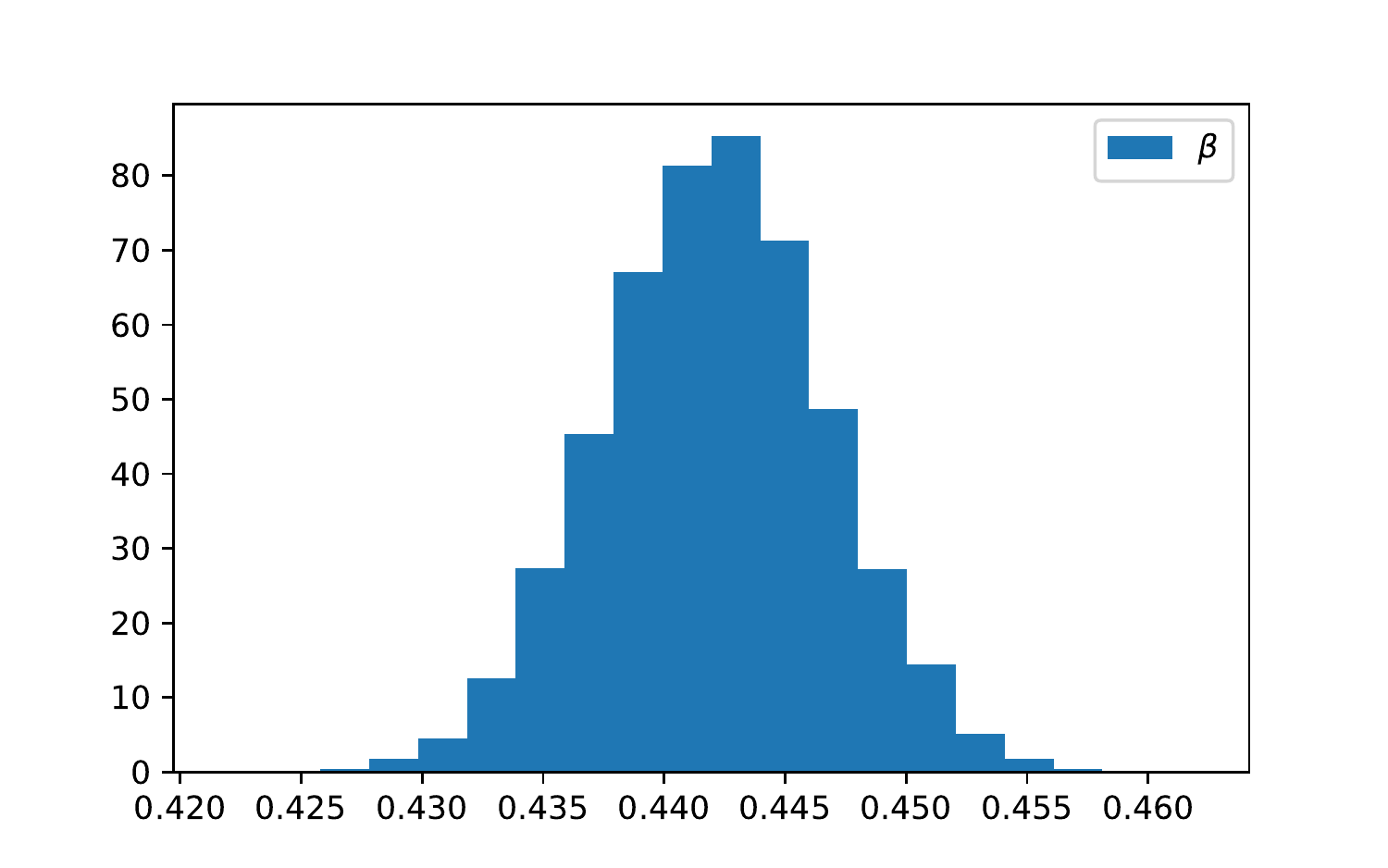}
\caption{Histograms of the posterior samples of $\alpha$ (left) and $\beta$ (right) for Example \ref{sumgamma}.}
\label{sumgamma:hist1b}
\end{center}
\end{figure}

\begin{figure}[htbp]
\begin{center}
\includegraphics[width=0.5\textwidth]{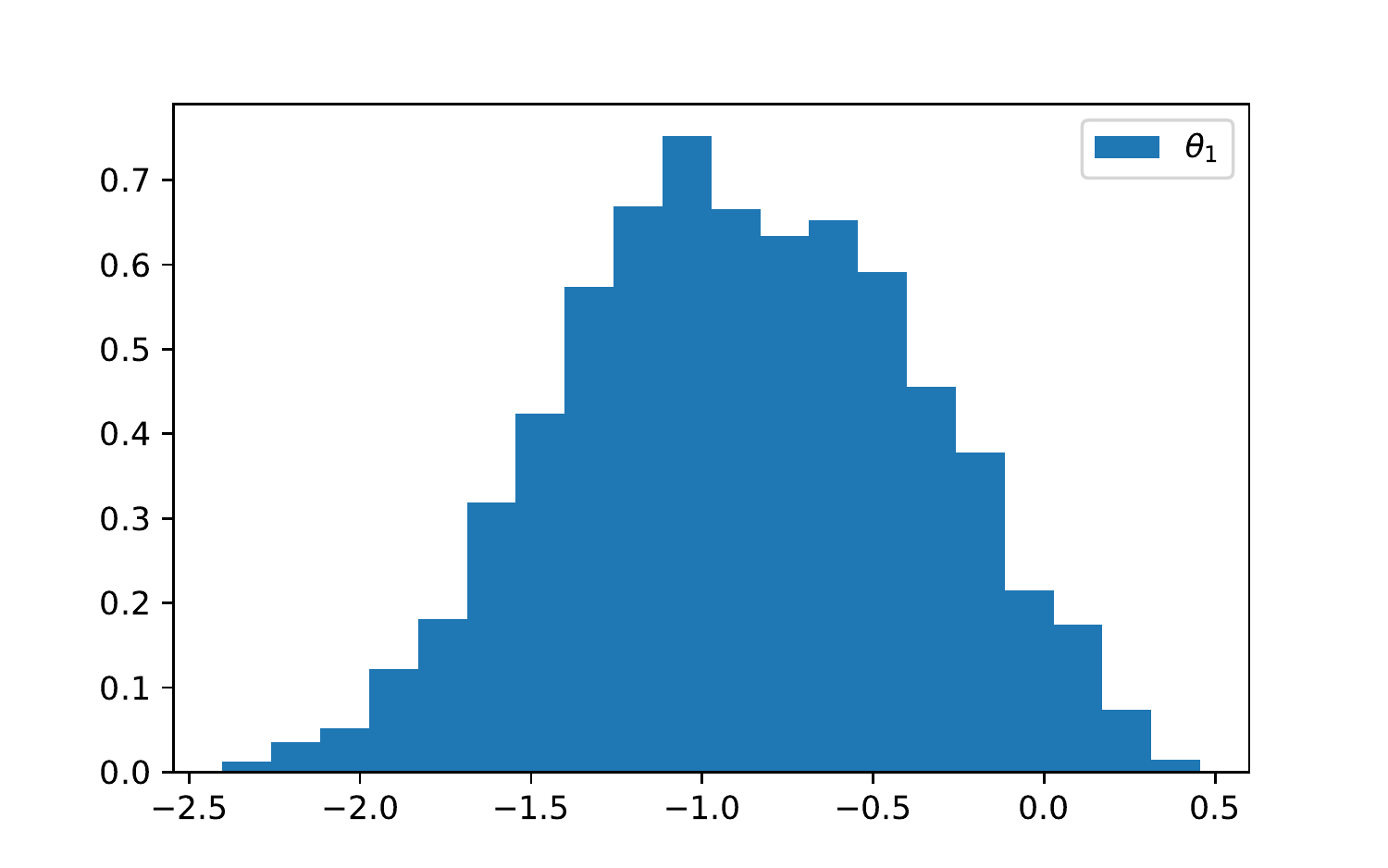}\includegraphics[width=0.5\textwidth]{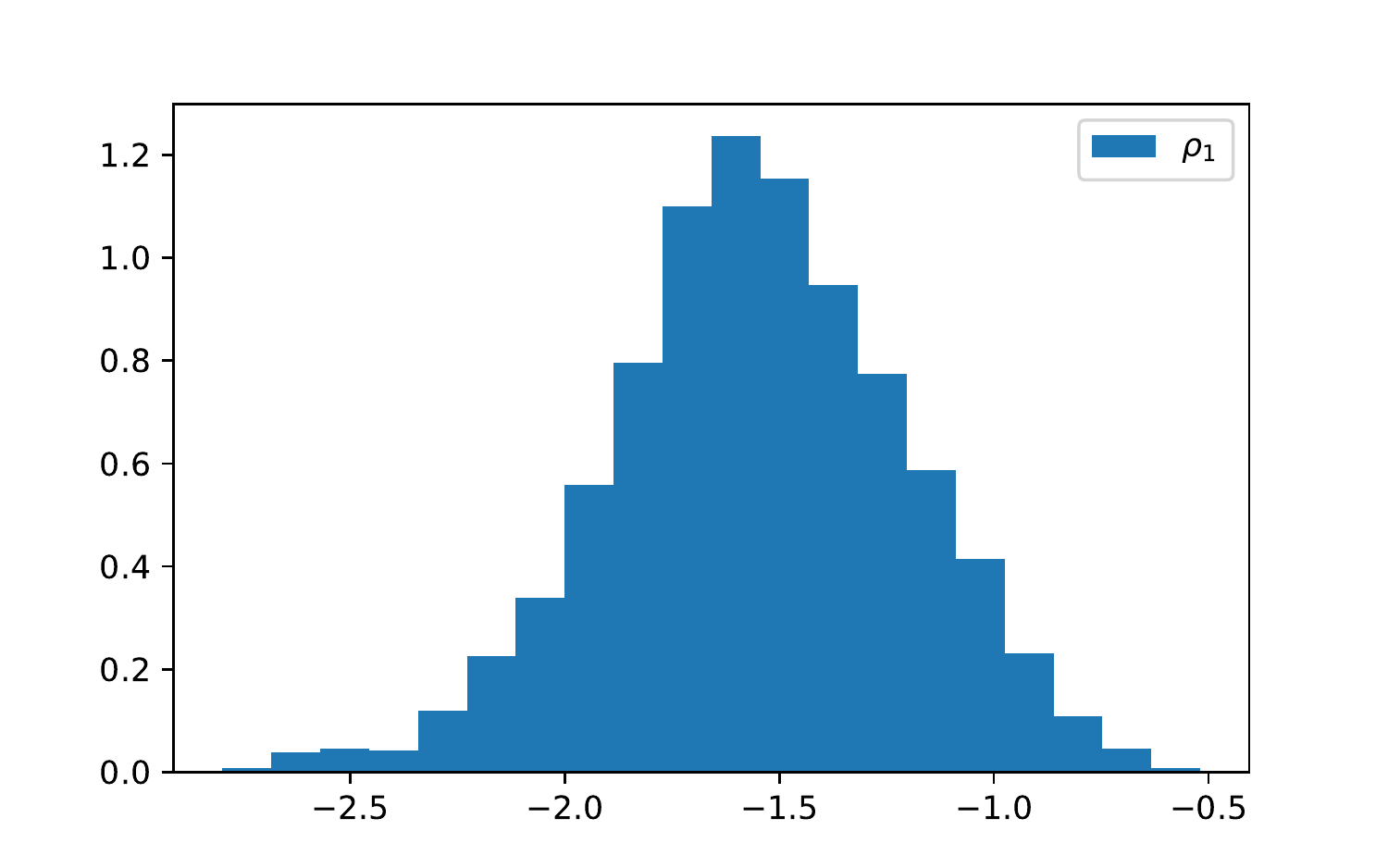}
\includegraphics[width=0.5\textwidth]{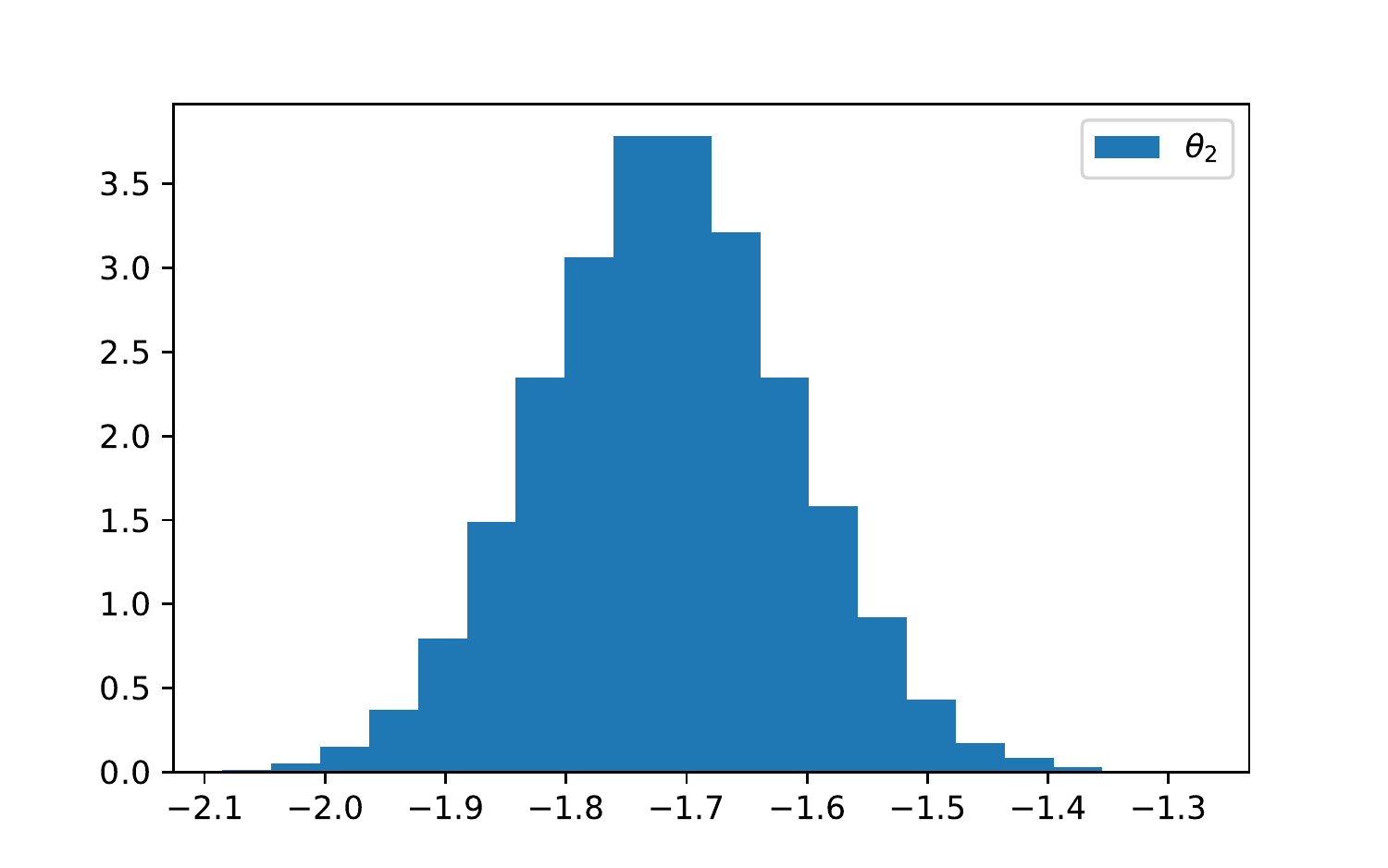}\includegraphics[width=0.5\textwidth]{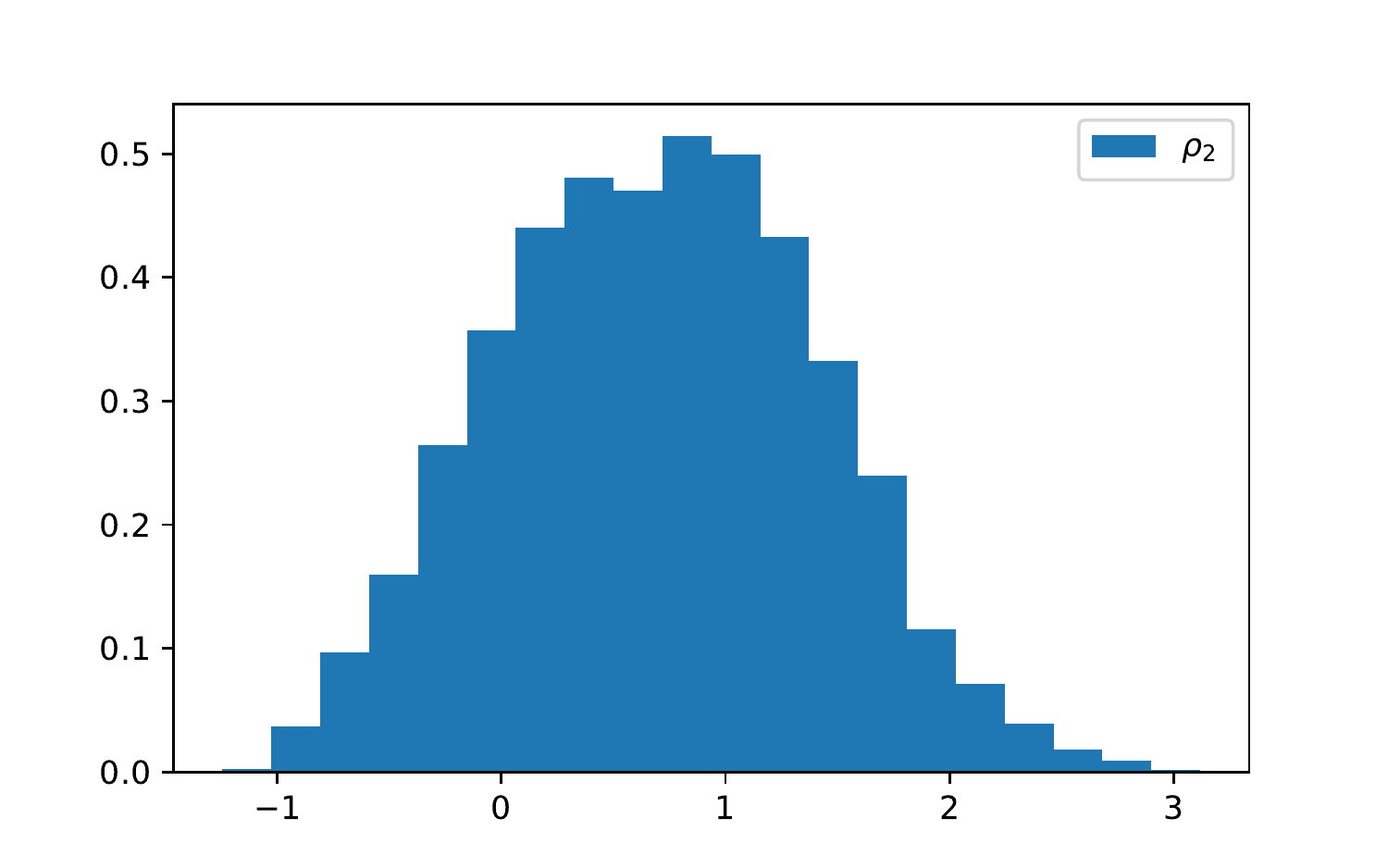}
\includegraphics[width=0.5\textwidth]{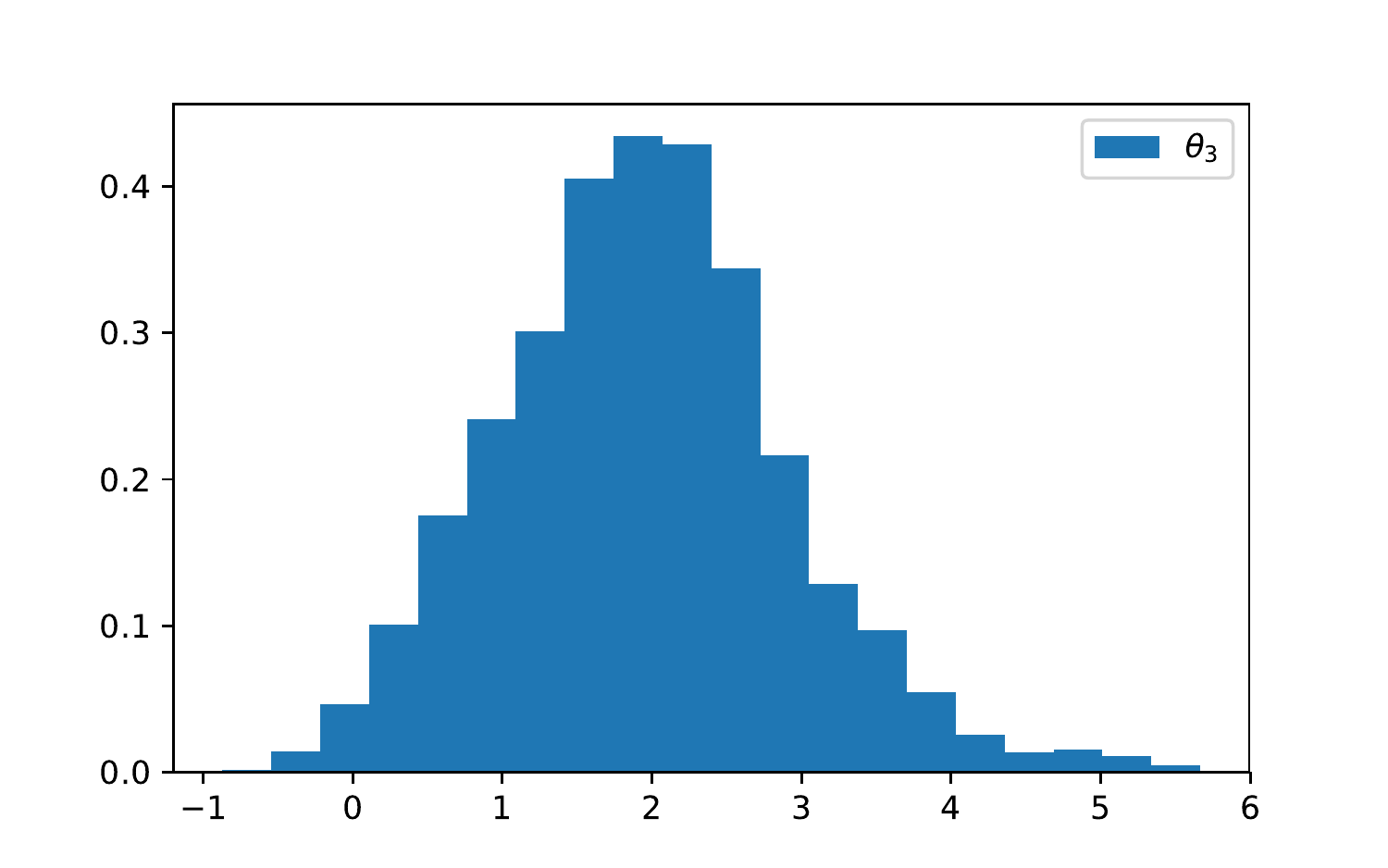}\includegraphics[width=0.5\textwidth]{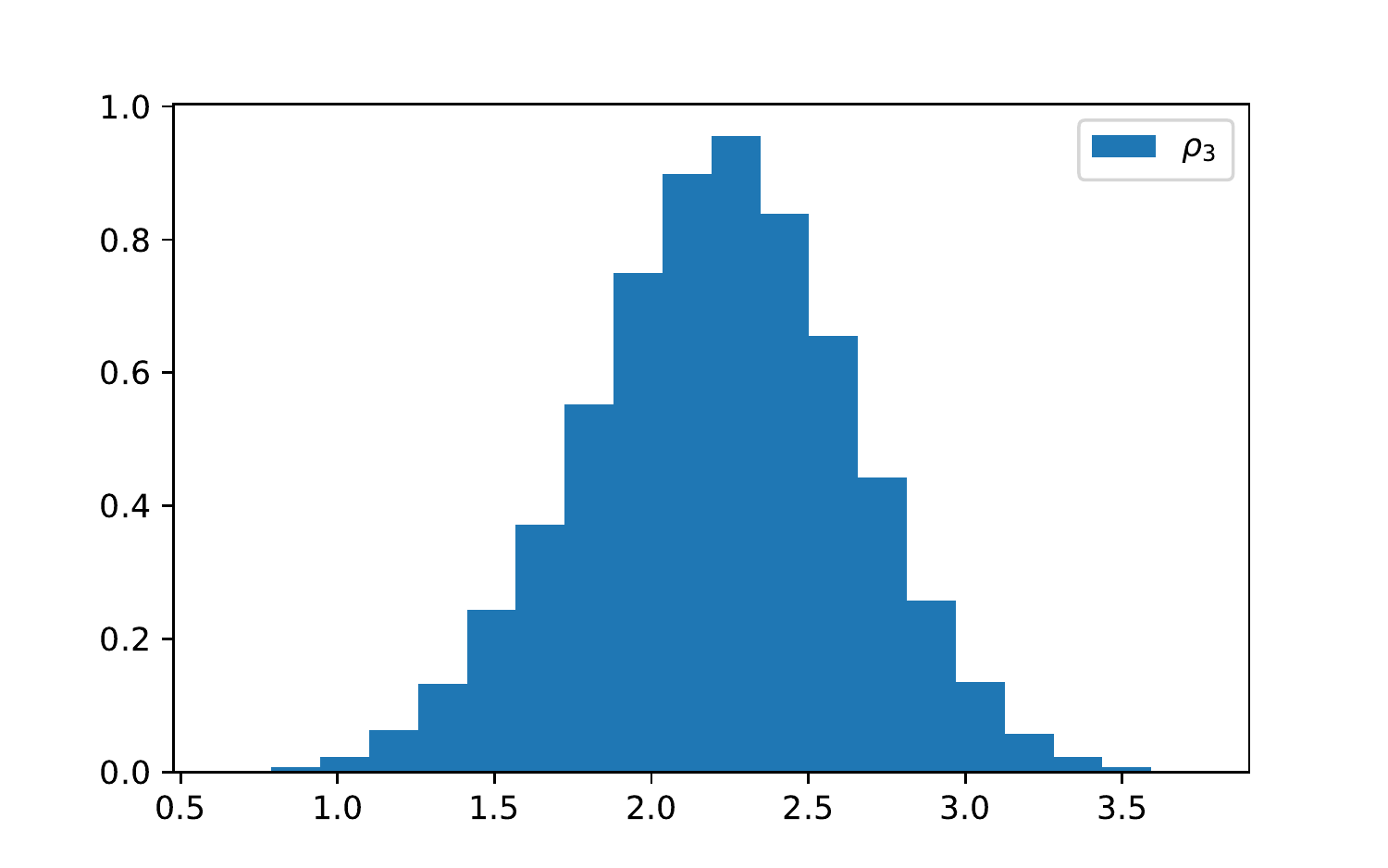}
\caption{Histograms of the posterior samples of the parameters for Example \ref{sumgamma}.  Left column: parameters $\theta_1$, $\theta_2$, $\theta_3$.
Right column: parameters $\rho_1$, $\rho_2$, $\rho_3$.
}
\label{sumgamma:hist2b}
\end{center}
\end{figure}

\begin{figure}[htbp]
\begin{center}
\includegraphics[width=\textwidth]{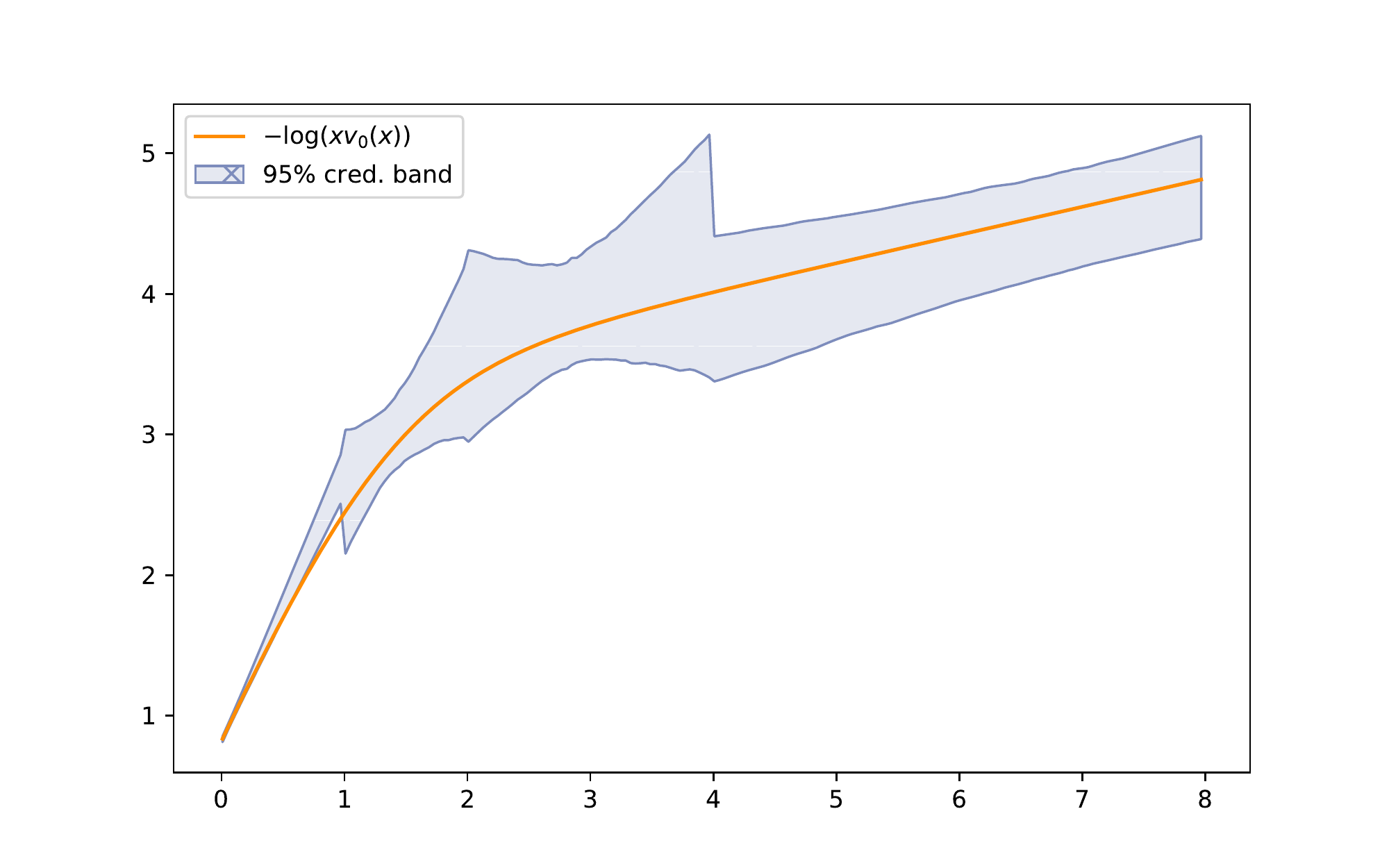}
\caption{Marginal Bayesian credible bands for Example \ref{sumgamma} for the function $-\log(x v(x))$ based on all samples.
Red: truth $-\log(x v_0(x))$ from equation \eqref{sg:truth}.
}
\label{sumgamma:bandsb}
\end{center}
\end{figure}

\section{Danish data on fire losses}
\label{sec:danish}

Over the last two decades there has been an increasing interest in applying Bayesian methods to insurance problems, see, e.g., \cite{hong17} and references therein. \cite{hong17b} apply a Dirichlet process mixture prior to model the density of insurance claim sizes, and provide motivation for using a nonparametric Bayesian approach in the actuarial science. In this section we will apply our Bayesian approach to the Danish data on large fire losses.  This dataset is a standard test example in extreme value theory, and from that point of view it has been a subject of several deep studies, such as \cite{mcneil97} and \cite{resnick97}. Our goals here are more modest, and aim at demonstrating the facts that firstly, $\theta$-subordinators can be potentially used to capture some aggregate features of the Danish data on large fire losses, and secondly, statistical inference for real data modelled through such processes can be successfully performed using the Bayesian methodology developed in this paper. This can be viewed as a partial empirical investigation of the risk model based on Gamma processes from \cite{dufresne91}. As observed in \cite{hewitt79}, a single standard distribution, such as the gamma, log-gamma or log-normal distribution, may not suffice to adequately model the distribution of individual insurance losses. For instance, multimodality in claim size distribution may result from presence of hidden factors or due to existence of illegal practices, such as exaggeration of injuries and excessive treatment costs, that are well-documented in auto insurance; see, e.g., \cite{rempala05} and the references therein. Since allowing for greater flexibility, in particular multimodality, in claim size distribution modelling is likely to result in multimodality of marginal distributions of the cumulative risk process, using a $\theta$-subordinator instead of a Gamma process to model evolution of the cumulative risk process over time a priori appears to be a sound approach.

\subsection{Data description and visualisation}

A succinct description of the Danish data on large fire losses can be found on p.~298 in \cite{embrechts97}. The dataset (scaled for privacy reasons) comprises 2167 fire losses (adjusted suitably for inflation to reflect the 1985 values) in Denmark over the 10 year period starting on 6 January 1980 and ending on 30 December 1990, that exceed in size one million DKK, and that were registered by Copenhagen Reinsurance. The rationale for thresholding losses at one million DKK is given in \cite{mcneil97}, pp.~119--120, and consists in the fact that in practice it is virtually impossible to collect exhaustive data on small losses: insurance is typically provided against significant losses, while small losses are dealt with by insured parties directly.

The data can be accessed through the {\bf QRM} package in {\bf R} under the name {\texttt{danish}}. The time plot of the data is given in the left panel of Figure \ref{fig:danishdata}. Presence of several exceedingly large losses is apparent from the plot, and therefore we use a logarithmic transformation to stabilise extreme variations in the data. Furthermore, this transforms observations on $[1,\infty)$ to observations on $[0,\infty)$, the support of the marginal distributions of a $\theta$-subordinator. One feature of the data is that on numerous days no losses have been registered. This is not compatible with the behaviour of an infinite activity subordinator; in fact, such a subordinator $X$ with probability one must have an infinite number of jumps in every finite time interval, and hence its increments must be strictly positive with probability one. A simple fix to this is to aggregate log losses over longer time periods than daily ones; aggregation over weekly periods (from Monday to Sunday) turned out to be sufficient (except few cases, where we had to aggregate data over periods of two weeks). The aggregated data on a logarithmic scale is displayed in the right panel of Figure \ref{fig:danishdata}. The idea of aggregation is a natural one, and embodies the fact that a probabilistic model unsuitable on a certain time scale may very well be appropriate on another time scale. In fact, already Albert Einstein in his classical paper on the Brownian motion observed that his model for displacement of a Brownian particle becomes inapplicable as the time interval between successive measurements of positions of a Brownian particle becomes increasingly small; see pp.~380--381 in \cite{einstein06}.

\newcommand{\loc}{./img/}
\begin{figure}
\includegraphics[width=0.45\textwidth]{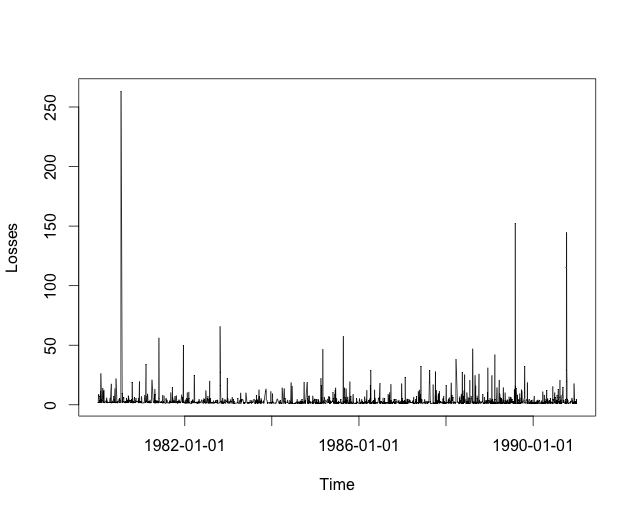}
\includegraphics[width=0.45\textwidth]{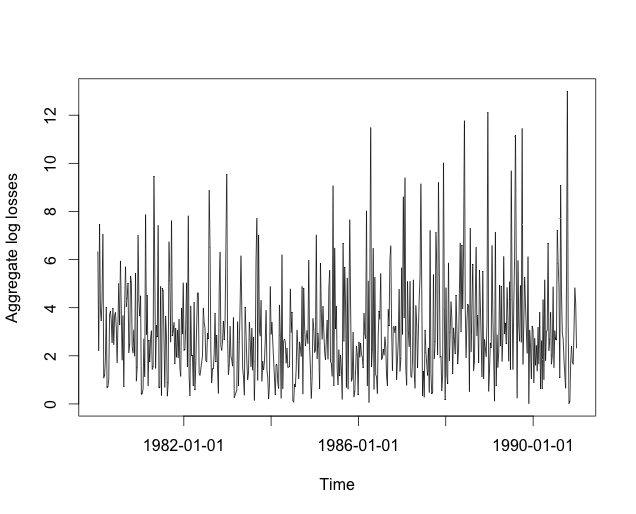}
\caption{Danish data on large fire losses. {\em Left}: original daily data (the unit is one million DKK). {\em Right}: logarithmically transformed and aggregated data.}
\label{fig:danishdata}
\end{figure}

According to the exploratory analysis of the transformed data that we supply in Appendix \ref{app:danish}, the data can be modelled as an i.i.d.\ sequence that follows a Gamma-like distribution, but perhaps is not genuinely Gamma. This suggests a possibility of using a $\theta$-subordinator to model the data.

\subsection{Modelling fire losses with a $\theta$-subordinator}

\begin{figure}[htbp]
\begin{center}
\includegraphics[width=\textwidth]{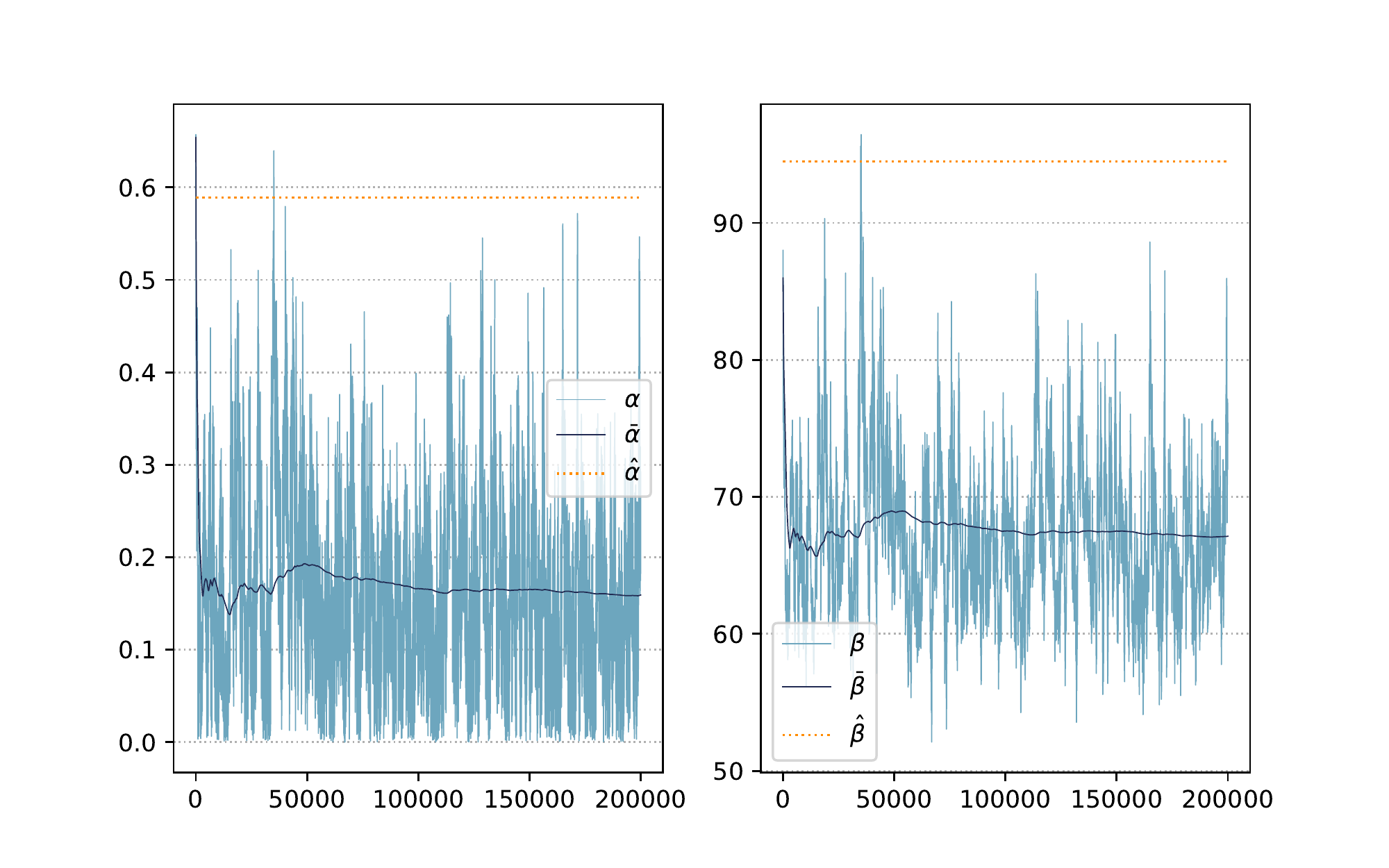}
\caption{Trace plots of the parameters $\alpha$ and $\beta$ for the fire loss data. Left: trace and running average of samples of $\alpha$.
(The latter indicated by $\bar \alpha$.) 
The maximum likelihood estimate $\hat\alpha$ of $\alpha$ using a Gamma process model is marked as the dotted yellow line.
Right: trace and running average of samples of $\beta$. (The latter indicated by $\bar \beta$.) The maximum likelihood estimate $\hat\beta$ of $\beta$ using a Gamma process model is marked as the yellow dotted line.
}
\label{danish:traceplot1b}
\end{center}
\end{figure}

\begin{figure}[htbp]
\begin{center}
\includegraphics[width=\textwidth]{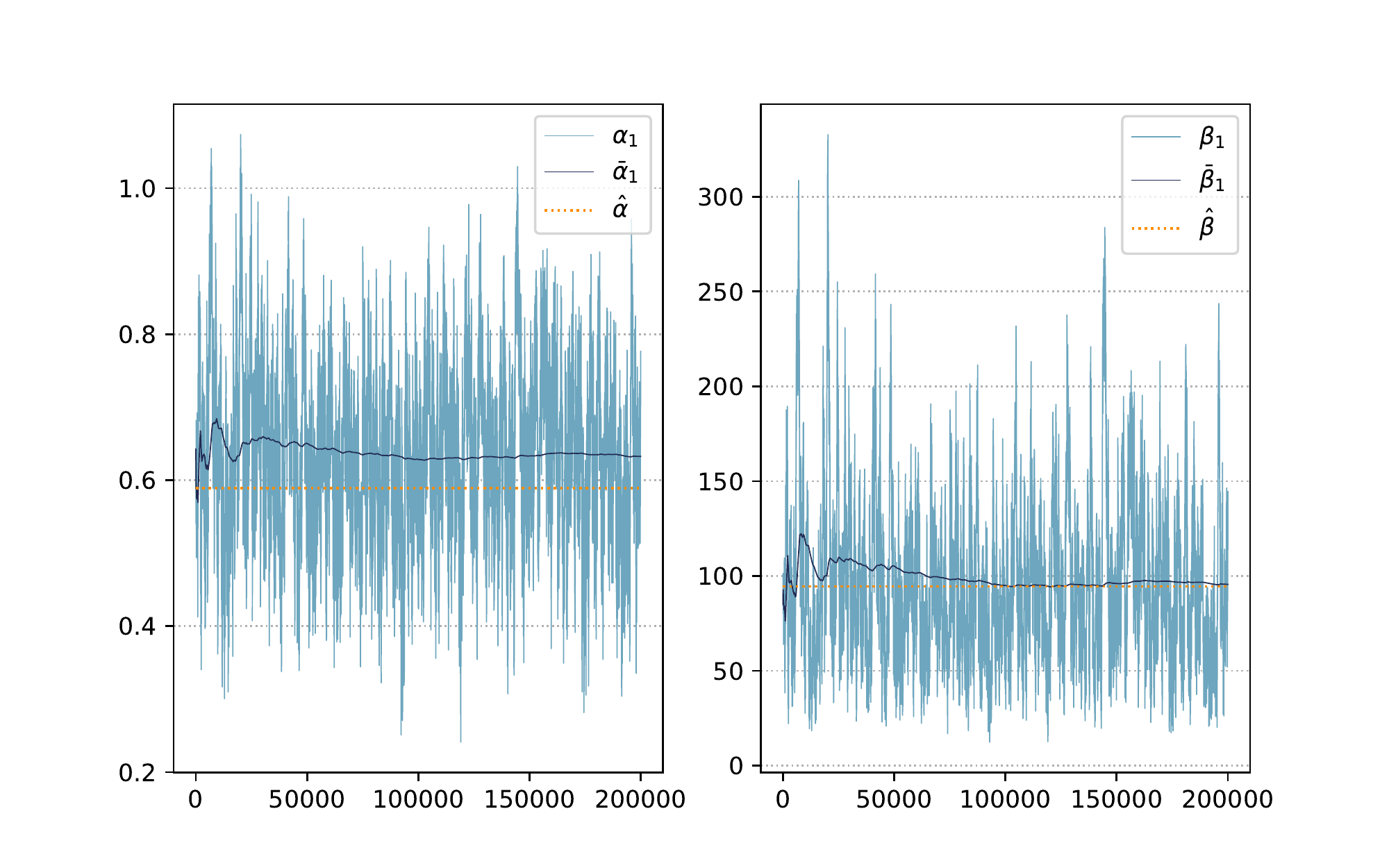}
\caption{Trace plots of the parameters used for the bin $(b_1, \infty)$ for the fire loss data. Left: trace and running average of the samples of $\alpha_1$.
The maximum likelihood estimate of $\alpha$ using a Gamma process model is marked as yellow line.
Right: trace and running average of the samples of $\beta_1$. 
The maximum likelihood estimate of $\beta$ using a Gamma process model is marked as yellow line.
}
\label{danish:traceplot2b}
\end{center}
\end{figure}

\begin{figure}[htbp]
\begin{center}
\includegraphics[width=\textwidth]{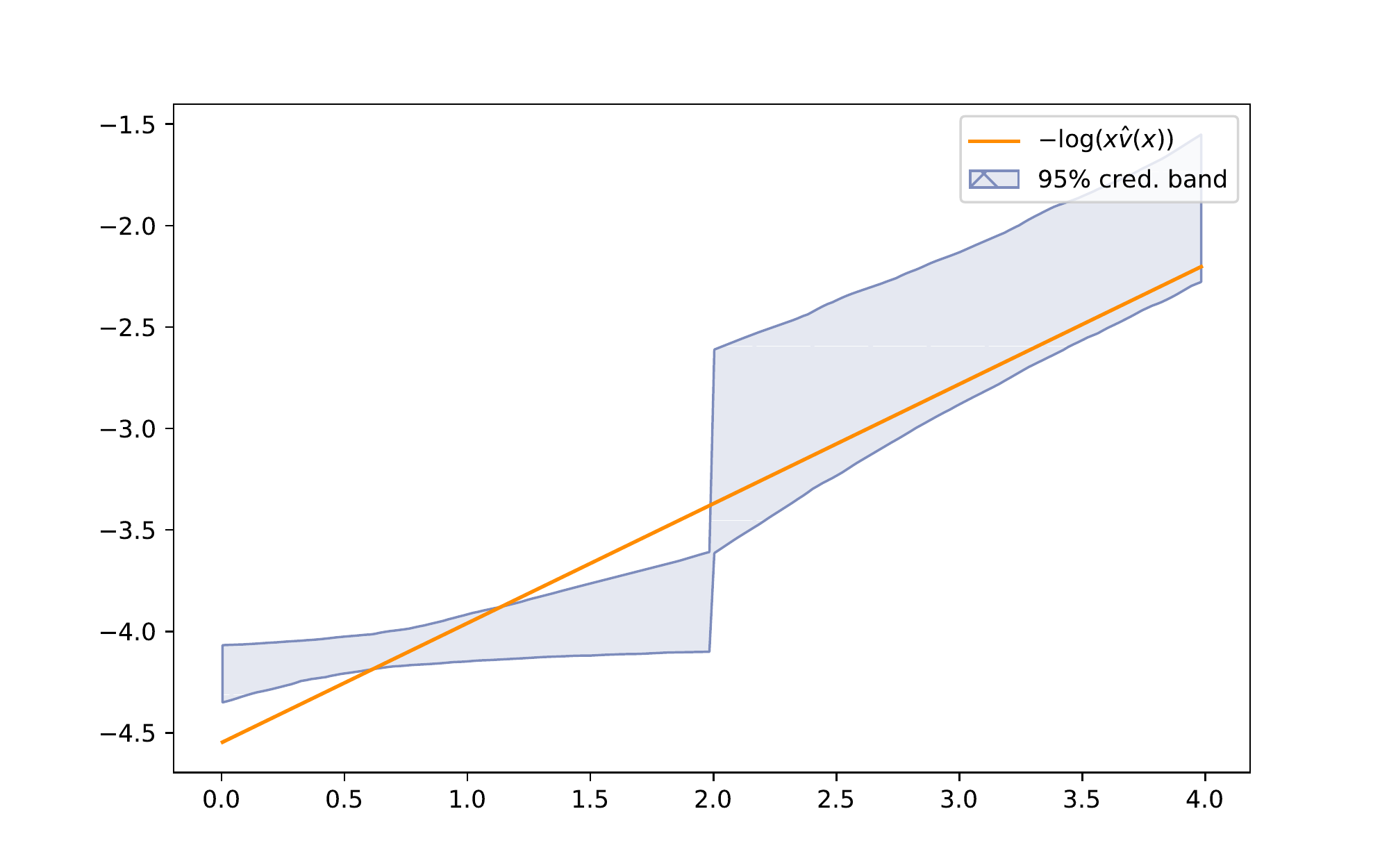}
\caption{Marginal Bayesian credible bands for the fire loss data for the function $-\log(x v(x))$ based on all samples.
Yellow: maximum likelihood estimate $-\log(x \hat v(x))$ assuming a Gamma process.
}
\label{danish:bandsb}
\end{center}
\end{figure}

Because the sample size is much smaller compared to our simulation examples, we chose $N=1$ corresponding to a single grid point $b_1 = 2$ and four parameters $\alpha$, $\beta$, $\theta_1$, $\rho_1$. In light of Example~\ref{sumgamma} and in order to improve mixing of the chain, we use a reparameterisation
$\alpha_1 = \alpha + \theta_1$, $\beta_1 = \beta\exp(-\rho)$, and work with four parameters  $\alpha$, $\beta$, $\alpha_1$, $\beta_1$, so that 
\[
v(x) = \begin{cases}\frac{\beta}x\exp(-\alpha x) &x \le b_1,\\
\frac{\beta_1}x\exp(-\alpha_1 x) &x > b_1.
\end{cases}
\]
A priori we equip these four parameters with independent Gamma distributions, with mean 0.75 and variance $0.36$ for  the parameters $\alpha, \alpha_1$, and mean $90$ and variance $2500$ for the parameters $\beta$, $\beta_1$.
In the data augmentation step we take intermediate points at distance $0.0192$, corresponding to $m = 1000$. 

For the parameter updates we took independent Gaussian innovations with standard deviations
$\sigma_\alpha = \sigma_{\alpha_1} = 0.03$, $\sigma_\beta = 1$ and $\sigma_{\beta_1} = 6$, respectively.
In the Gibbs sampler in each step new Gamma bridges are proposed  in the data augmentation step, followed by a parameter update Metropolis-Hastings step cycling through updates of $\beta$ in the first and second and the other parameters jointly in each of the remaining three of in total 5 stages. 
With these choices, the chains mix sufficiently well. 
The MCMC algorithm was run for $200\,000$ iterations. 
Figure~\ref{danish:traceplot1b}
shows trace plots and running averages of the posterior samples of the 
parameters $\alpha$ and $\beta$, whereas Figure~\ref{danish:traceplot2b} shows similar plots for the parameters $\alpha_1$ and $\beta_1$.

Figure~\ref{danish:bandsb} shows the $95\%$ marginal Bayesian credible band for the function $\theta(x)+\alpha x$ contrasted to the maximum likelihood estimate that assumes the observations come from a Gamma process. This plot suggests that modelling the losses with a Gamma process leads to overestimation of the number of small jumps and possibly of large jumps too; however, more data is necessary to make a definitive statement (unfortunately, as observed in \cite{chavez16}, it is difficult for academia to gain access to the insurance data).
In connection to this, we note that a difference in the estimates of the rate of decay of the L\'evy density (value of $\alpha_1$ in the model) has serious implications of practical relevance for the assessment of the risk of very large fire losses.

\section{Outlook}
\label{sec:outlook}
As a possible extension of the model studied in this paper, one can consider a class of increasing, infinite activity L\'evy processes, which one can call $(a,b, \theta)$-subordinators. 
Fix some \(a\in [0,1),\) \(b\geq0\) and a non-decreasing, non-negative function  \(\theta\) on \(\mathbb{R}_+\); then a L\'evy process \((X_t)_{t\geq 0}\) is called  an $(a,b,\theta)$-subordinator, if the characteristic function of  \(X_1\) has the form 
\begin{eqnarray*}
\varphi(z)=\mathrm{E}[e^{izX_1}]=\exp\left(\int_{\mathbb{R}} (e^{izx}-1)\,\nu(\dd x) \right),\quad z\in \mathbb{R},
\end{eqnarray*}
where the L\'evy measure \(\nu\) is given by
\begin{eqnarray}
\nu(\dd x)=\frac{b}{x^{1+a}}e^{-\theta(x)}\one_{(0,\infty)}(x)\,\dd x.
\end{eqnarray}
On one hand, this model generalises the Gamma process with \(a=0\) and \(\theta(x)\equiv \lambda x\), \(\lambda>0\). On the other hand, $(a,b,\theta)$-subordinators cover the class of one-sided tempered stable processes, that have recently gained attention in physics and mathematical finance, see \cite{rachev}.  Furthermore, the family of $(a,b,\theta)$-subordinators overlaps  with the class of self-decomposable L\'evy processes, that likewise have important applications in finance, see, e.g., \cite{carr07}. 
\par
In order to extend the inferential approach presented in the current work to this new model, we need to be able to sample from the distribution of $X$ 
 conditional on $X_{T} = x_T.$  
%
%
 The problem of sampling from tempered stable bridges has been recently studied in \cite{kim}. Let us also mention the fact that the problem of estimating the stability index \(\alpha\) is difficult from a Bayesian point of view due to singularity of  the measures induced by two L\'evy processes with different stability indices. However, several frequentist approaches to estimate \(\alpha \) are available in the literature, see \cite{belomestnyReiss2006}. Also, our estimation approach can be conceivably extended to Gamma driven stochastic differential equation models.

\appendix
\section{Technical results for Section \ref{section:consistency}}
\label{lemmata0}

\emph{Proof of Proposition~\ref{lem:gnedenko}.}

For ease of notation we put $\mu_n(\dd x)=(x^2 \wedge 1)\nu_{n}(\dd x)$ and $\mu(\dd x)= (x^2 \wedge 1)\nu(\dd x)$. Gnedenko's theorem, see, e.g., Theorem 2 in \cite{gnedenko39}, states that $\mathbb{Q}_{v_n} \rightsquigarrow \mathbb{Q}_{v}$ if and only if $\gamma_n\to \gamma$ and $\mu_n\rightsquigarrow \mu$, referred in this proof as Gnedenko's conditions. We show that these conditions are equivalent to $\widetilde{\nu}_n \rightsquigarrow \widetilde{\nu}$. Assume the latter and take the bounded and continuous function $f=1$. It then follows that $\gamma_n+\mu_n(\RR)\to\gamma+\mu(\RR)$. Next we show that $\gamma_n\to\gamma$. Let $f_\eps(x)=(1-\frac{x}{\eps})^+$ for $x\geq 0$ and $0<\eps\leq 1$. Then 
\[
0\leq\int f_\eps(x)\mu_n(\dd x)=\int_0^\eps f_\eps(x)x^2\nu_n(\dd x)\leq \int_0^\eps x^2\nu_n(\dd x)\leq \eps\int_0^\eps x\nu_n(\dd x)\leq\eps\gamma_n.\]
 It follows  that $\gamma_n\leq\int f_\eps\dd\widetilde\nu_n\leq (1+\eps)\gamma_n$, and hence $\limsup \gamma_n\leq \int f_\eps\dd\widetilde\nu\leq (1+\eps)\liminf \gamma_n$. Similar considerations yield $\gamma\leq \int f_\eps\dd\widetilde\nu\leq (1+\eps)\gamma$, and a combination of these results yields $\max\{\limsup \gamma_n,\gamma\}\leq  (1+\eps)\min\{\liminf \gamma_n, \gamma\}$. Since $\eps$ is arbitrary, it follows that $\gamma_n\to\gamma$ and, in view of the earlier limit, also $\mu_n(\RR)\to \mu(\RR)$. Let $f_0$ be bounded and continuous such that $f_0(0)=0$. Then $\int f_0\dd\mu_n=\int f_0\dd\widetilde\nu_n\to \int f_0\dd\widetilde\nu=\int f_0\dd\mu$. Take now an arbitrary bounded and continuous function $f$, and let $f_0=f-f(0)$. Then, in view of the above, one has  $\int f \dd \mu_n=\int f_0 \dd\mu_n+f(0)\mu_n(\RR)\to\int f_0 \dd\mu+f(0)\mu(\RR)=\int f\dd\mu$. Both of Gnedenko's conditions are thus satisfied. This shows one implication. Conversely, by assuming Gnedenko's conditions, one obtains by a simple addition that $
\widetilde{\nu}_n \rightsquigarrow \widetilde{\nu}$.
\endproof

The next two lemmas bound the Kullback-Leibler divergence between two measures $\mathbb{Q}_{v_0},\mathbb{Q}_{v}.$

\begin{lemma}
\label{lem:kl}
We have $\mathcal{KL}(\mathbb{Q}_{v_0},\mathbb{Q}_{v}) \leq \mathcal{KL}(\mathbb{P}_{v_0},\mathbb{P}_{v})$.
\end{lemma}

\begin{proof}
This is the inequality stated on p.~12 in \cite{cpp15}. The fact that there it is obtained in the context of the compound Poisson processes plays no role in our case: the result follows from the well-known inequality due to \cite{csiszar63}; cf.\ Lemma 2 and arguments preceding it in \cite{cpp15}.
\end{proof}

\begin{lemma}
\label{lem:kl2}
We have
$
\mathcal{KL}(\mathbb{Q}_{v_0},\mathbb{Q}_{v}) \lesssim |\alpha-\alpha_0|+\| \theta - \theta_0 \|_{\infty}.
$
The constant in the inequality depends on $\alpha_0,\theta_0$ and known constants only.
\end{lemma}

\begin{proof}
We will bound from above $\mathcal{KL}(\mathbb{P}_{v_0},\mathbb{P}_{v})$, which by Lemma \ref{lem:kl} automatically yields an upper bound on $\mathcal{KL}(\mathbb{Q}_{v_0},\mathbb{Q}_{v}).$ By formula (A.1) in \cite{cont06},
\[
\mathcal{KL}(\mathbb{P}_{v_0},\mathbb{P}_{v})=\int_{x>0} v_0(x) \log\left(\frac{v_0(x)}{v(x)}\right)\dd x + \int_{x>0} (v(x)-v_0(x))\dd x = \textrm{I}+\textrm{II}.
\]
We will separately bound the two terms. We start with the first one:
\[
\textrm{I} = (\alpha-\alpha_0) \int_{x>0} e^{-\alpha_0 x-\theta_0(x)} \dd x + \int_{\underline{b}\leq x \leq \overline{b}} \frac{1}{x} e^{-\alpha_0 x-\theta_0(x)} (\theta(x)-\theta_0(x)) \dd x.
\]
It follows that
$
|\textrm{I}| \lesssim |\alpha-\alpha_0| + \| \theta - \theta_0 \|_{\infty}.
$
The constant in the inequality depends on $\alpha_0,\theta_0,$ and known constants.

Now we turn to $\textrm{II}.$ We have
\begin{align*}
\textrm{II} & = \int_{0 < x < \underline{b}} \frac{1}{x} \left( e^{-\alpha x} - e^{-\alpha_0 x} \right)\dd x\\
&+ \int_{\underline{b}\leq x \leq \overline{b}} \frac{1}{x} \left( e^{-\alpha x-\theta(x)} - e^{-\alpha_0 x - \theta_0(x)} \right)\dd x \\
&+ \int_{\overline{b} < x <\infty} \frac{1}{x} \left( e^{-\alpha x} - e^{-\alpha_0 x} \right)\dd x.
\end{align*}
By the mean-value theorem, using also the facts that $\alpha_0,\alpha \geq \underline{\alpha},$ $x>0,$ the first term on the right in the above display is up to a constant bounded in absolute value by $|\alpha-\alpha_0|.$ A similar bound is true for the third term too. As far as the second term is concerned, notice that for any $x,y,$
\[
|e^x-e^y| \leq \max (e^x,e^y) |x-y|,
\]
so that for $x\in [\underline{b},\overline{b}]$ we have
\[
\left | e^{-\alpha x-\theta(x)} - e^{-\alpha_0 x - \theta_0(x)}  \right| \lesssim |\alpha-\alpha_0| x +\| \theta - \theta_0 \|_{\infty}.
\]
This in turn entails that
\[
\left| \int_{\underline{b}\leq x \leq \overline{b}} \frac{1}{x} \left( e^{-\alpha x-\theta(x)} - e^{-\alpha_0 x - \theta_0(x)} \right)\dd x \right| \lesssim |\alpha-\alpha_0| +\| \theta - \theta_0 \|_{\infty}.
\]
Combination of the above intermediate inequalities completes the proof.
\end{proof}

The next three lemmas bound the discrepancy $\mathcal{V}$ between two measures $\mathbb{Q}_{v_0},\mathbb{Q}_{v}.$

\begin{lemma}
\label{lem:v}
We have
\[
\mathcal{V}(\mathbb{Q}_{v_0},\mathbb{Q}_{v}) \leq \mathcal{V}(\mathbb{P}_{v_0},\mathbb{P}_{v}) + 4\mathcal{KL}(\mathbb{P}_{v_0},\mathbb{P}_{v}).
\]
\end{lemma}

\begin{proof}
This is equation (21) in \cite{cpp15}. The fact that in the original context it dealt with the compound Poisson process, plays no role in our case, the arguments go through without modification.
\end{proof}

\begin{lemma}
\label{lem:v2}
We have
\begin{multline*}
\mathcal{V}(\mathbb{P}_{v_0},\mathbb{P}_{v}) = \int_{0}^{\infty}v_{0}(y)\log^{2}\left(\frac{v(y)}{v_{0}(y)}\right)\,\dd y\\
+\left(\int_{0}^{\infty}\left(1-\frac{v(y)}{v_{0}(y)}+\log\left(\frac{v(y)}{v_{0}(y)}\right)\right)v_{0}(y)\,\dd y\right)^{2}.
\end{multline*}
\end{lemma}

\begin{proof}
It follows from Theorem 4 in \cite{brockett78} that
\begin{align*}
\phi(u)&\coloneqq \operatorname{E}_{\mathbb{P}_{v_{0}}}\left[\exp\left(\ii u\log\left(\frac{\dd\mathbb{P}_{v}}{\dd\mathbb{P}_{v_{0}}}\right)\right)\right]
\\
&=\exp\left[\ii u\int_{0}^{\infty}\left(1-\frac{v(x)}{v_{0}(x)}\right)v_{0}(x)\,\dd x+\int_{0}^{\infty}\left(e^{\ii ux}-1\right)v_{0}\circ g^{-1}(\dd x)\right]
\end{align*}
with $g(x)=\log\left(\frac{v(x)}{v_{0}(x)}\right).$
We have 
\[
\phi'(u)=\left(\ii \int_{0}^{\infty}\left(1-\frac{v(x)}{v_{0}(x)}\right)v_{0}(x)\,\dd x+\ii \int_{0}^{\infty}xe^{\ii ux}(v_{0}\circ g^{-1})(\dd x)\right)\phi(u)
\]
and 
\begin{align*}
\phi''(u) & =-\left(\int_{0}^{\infty}x^{2}e^{\ii ux}(v_{0}\circ g^{-1})(\dd x)\right)\phi(u)\\
 & -\left(\int_{0}^{\infty}\left(1-\frac{v(x)}{v_{0}(x)}\right)v_{0}(x)\,\dd x+\int_{0}^{\infty}xe^{\ii ux}(v_{0}\circ g^{-1})(\dd x)\right)^{2}\phi(u).
\end{align*}
As a result, we get that
\begin{multline*}
\operatorname{E}_{\mathbb{P}_{v_{0}}}\left[\left(\log\left(\frac{\dd\mathbb{P}_{v}}{\dd\mathbb{P}_{v_{0}}}\right)\right)^{2}\right]=-\phi^{\prime\prime}(0)
=\int_{0}^{\infty}x^{2}(v_{0}\circ g^{-1})(\dd x)
\\
+\left(\int_{0}^{\infty}\left(1-\frac{v(x)}{v_{0}(x)}\right)v_{0}(x)\,\dd x+\int_{0}^{\infty}x(v_{0}\circ g^{-1})(\dd x)\right)^{2}.
\end{multline*}
Now note that by the change of the variable formula,
\begin{align*}
\int_{0}^{\infty}x(v_{0}\circ g^{-1})(\dd x)&=\int_{0}^{\infty}v_{0}(y)\log\left(\frac{v(y)}{v_{0}(y)}\right)\,\dd y,\\
\int_{0}^{\infty}x^{2}(v_{0}\circ g^{-1})(\dd x)&=\int_{0}^{\infty}v_{0}(y)\log^{2}\left(\frac{v(y)}{v_{0}(y)}\right)\,\dd y.
\end{align*}
This completes the proof.
\end{proof}

The next result is used to bound from below the denominator in the posterior and is a simple restatement of Lemma 8.1 in \cite{ghosal00}.

\begin{lemma}
\label{lem:denominator}
Let $\widetilde{\Pi}$ be an arbitrary probability measure on the set
\[
K(\delta)=\{ v\colon  \mathcal{KL}(\mathbb{Q}_{v_0},\mathbb{Q}_{v}) \leq \delta , \mathcal{V}(\mathbb{Q}_{v_0},\mathbb{Q}_{v}) \leq \delta \},
\]
where $\delta>0$ is any fixed number. Then for every constant $C>1,$
\[
\mathbb{Q}_{v_0}^n\left( \int_{ K(\delta) } \prod_{i=1}^n \frac{\dd\mathbb{Q}_{v}}{\dd\mathbb{Q}_{v_0}}(Z_i)\widetilde{\Pi}(\dd v) \leq e^{-Cn\delta} \right) \leq \frac{1}{(C-1)^2n\delta}. 
\]
\end{lemma}

The next lemma, together with Lemma \ref{lem:kl2}, is instrumental in verifying the prior mass condition, that is one of the key ingredients for derivation of posterior consistency.

\begin{lemma}
\label{lem:v3}
We have
\[
\mathcal{V}(\mathbb{Q}_{v_0},\mathbb{Q}_{v}) \lesssim |\alpha-\alpha_0|+\| \theta - \theta_0 \|_{\infty} + |\alpha-\alpha_0|^2+\| \theta - \theta_0 \|_{\infty}^2.
\]
The constant in the inequality depends on $\alpha_0,\theta_0$ and known constants only.
\end{lemma}

\begin{proof}
The result follows from Lemmas \ref{lem:kl2}, \ref{lem:v} and \ref{lem:v2} after some tedious calculations as in the proof of Lemma \ref{lem:kl2}.
\end{proof}

The next results deals with the prior mass condition.

\begin{lemma}
\label{lem:prior}
For every $\delta>0$ small enough and all $n$ large,
\[
\Pi_n\left( K(\delta) \right) \gtrsim (c \delta)^{2N_n}
\]
for a constant $c$ independent of $n.$
\end{lemma}

\begin{proof}
By Lemmas \ref{lem:kl2} and \ref{lem:v3}, there exists a constant $c>0,$ such that
\[
K(\delta) \subseteq \{ |\alpha-\alpha_0| \vee |\alpha-\alpha_0|^2 \leq c\delta \} \cap \{  \|\theta-\theta_0\|_{\infty} \vee \|\theta-\theta_0\|_{\infty}^2 \leq c\delta \}.
\]
Since priors on $\alpha$ and $\theta$ are independent, we get that
\begin{align*}
\Pi_n( K(\delta) ) &\geq \left [ \Pi_n( |\alpha-\alpha_0| \leq c\delta ) \wedge \Pi_n( |\alpha-\alpha_0|^2 \leq c\delta ) \right ]\\
& \times \left [ \Pi_n( \|\theta-\theta_0\|_{\infty} \leq c\delta ) \wedge \Pi_n( \|\theta-\theta_0\|_{\infty}^2 \leq c\delta ) \right ] .
\end{align*}
We will bound each of the terms on the right separately. For $\delta$ small enough,
\[
\Pi_n( |\alpha-\alpha_0| \leq c\delta ) \leq \Pi_n( |\alpha-\alpha_0|^2 \leq c\delta ), \quad \Pi_n( \|\theta-\theta_0\|_{\infty} \leq c\delta ) \leq \Pi_n( \|\theta-\theta_0\|_{\infty}^2 \leq c\delta ),
\]
so that it is sufficient to bound from below the terms on the left hand side of these two inequalities.

Note that since $\alpha$ is equipped with the uniform prior, $\Pi_n( |\alpha-\alpha_0| \leq c\delta ) \asymp \delta.$ On the other hand,
\begin{align*}
\Pi_n( \|\theta-\theta_0\|_{\infty} \leq c\delta ) & =  \Pi_n\left( \max_{1 \leq k \leq N} \sup_{x \in {B_k}} |\theta(x)-\theta_0(x)| \leq c\delta \right)\\
&=\prod_{k=1}^{N_n} \Pi_n\left(  \sup_{x \in {B_k}} |\theta(x)-\theta_0(x)| \leq c\delta \right).
\end{align*}
Consider a term
\[
\Pi_n\left(  \sup_{x \in {B_k}} |\theta(x)-\theta_0(x)| \leq c\delta \right) = \Pi_n\left(  \sup_{x \in {B_k}} |\rho_k+\theta_k x -\theta_0(x)| \leq c\delta \right).
\]
By the H\"older assumption on $\theta_0$,  we have by the triangle inequality
\begin{align*}
|\rho_k+\theta_k x -\theta_0(x)| & \leq |\rho_k+\theta_kb_k-\theta_0(b_k)|+ L(x-b_k)^{\lambda} \\
& \leq |\rho_k+\theta_kb_k-\theta_0(b_k)|+L\Delta_n^{\lambda},
\end{align*}
where $\Delta_n$ denotes the length of the bins, $\Delta_n=\overline{b}/N_n$.
As $\Delta_n\to 0$ for $n\rightarrow\infty$, we can make it small enough to have (for any $c,\delta>0$) $L\Delta_n^{\lambda}\leq\delta/2$. 
It follows that for sufficiently small $\delta$ one has
\[
\left\{  \sup_{x \in {B_k}} |\theta(x)-\theta_0(x)| \leq c\delta \}\supset 
\{ |\rho_k+\theta_kb_k - \theta_0(b_k)| \leq \frac{c\delta}{2} \right\}.
\]
Furthermore, 
we have
\[
\left\{ |\rho_k+\theta_k b_k - \theta_0(b_k)| \leq \frac{c\delta}{2} \right\} \supset \left\{ |\rho_k - \theta_0(b_k)| \leq \frac{c\delta}{4} \right\} \cap \left\{  |\theta_k b_k| \leq \frac{c\delta}{4} \right\}.
\]
Then by independence of $\theta_k$ and $\rho_k$,
\[
\Pi_n\left( |\rho_k+\theta_k b_k - \theta_0(b_k)| \leq \frac{c\delta}{2}  \right) \geq \Pi_n \left(  |\rho_k - \theta_0(b_k)| \leq \frac{c\delta}{4}  \right) \Pi_n\left( |\theta_k b_k| \leq \frac{c\delta}{4} \right).
\]
As the interval $(\theta_0(b_k)-\frac{c\delta}{4},\theta_0(b_k)+\frac{c\delta}{4})$ is contained in $[-\bar\theta,\bar\theta]$ for all sufficiently small $\delta$,
the first factor on the right is bounded from below by a constant (independent of $n$ and $k$) times $\delta$. So is the second factor, because
\[
\Pi_n\left( |\theta_k b_k| \leq \frac{c\delta}{4} \right) \geq \Pi_n\left( |\theta_k | \leq \frac{c\delta}{4\overline{b}} \right).
\]
It follows that
\[
\Pi_n\left(  \sup_{x \in {B_k}} |\theta(x)-\theta_0(x)| \leq c\delta \right) \gtrsim \delta^2.
\]
Thus, after an evident renaming of constants,
$
\Pi_n\left( K(\delta) \right) \gtrsim (c \delta)^{2N_n}
$
for a constant $c$ independent of $n$.
\end{proof}

The result of the next lemma is a variation on Lemma~\ref{lem:kl2}. Its main use lies in establishing a certain metric entropy bound in Lemma \ref{lemma:covering}.

\begin{lemma}\label{lemma:hell}
It holds that $d_{\mathcal{H}}(\mathbb{Q}_{v_0},\mathbb{Q}_v)\lesssim |\alpha-\alpha_0| + ||\theta-\theta_0||_\infty$.
\end{lemma}

\begin{proof}
We first note that
\[
d_{\mathcal{H}}^2(\mathbb{Q}_{v_0},\mathbb{Q}_v)  \leq d_{\mathcal{H}}^2(\mathbb{P}_{v_0}\mathbb{P}_{v}),
\]
see \cite{cpp15}, p.~14.

Further, one has $d_{\mathcal{H}}^2(\mathbb{P}_{v_0}\mathbb{P}_{v}) = 1-\exp(-h)  \leq h$, 
see Theor\`eme~1 in \cite{MS85},
where $h=\half \int_0^\infty (\sqrt{v_0(x)}-\sqrt{v(x)})^2\,\dd x$.
By a splitting procedure as in the proof of Lemma~\ref{lem:kl2}, we get $d_{\mathcal{H}}^2(\mathbb{Q}_{v_0},\mathbb{Q}_v)\lesssim |\alpha-\alpha_0|^2 + ||\theta-\theta_0||_\infty^2$. Finally, use the inequality $\sqrt{x^2+y^2}\leq x+y$ for $x,y\geq 0$.
\end{proof}

In the proof of Lemma~\ref{lem:maximal} below we need an auxiliary result. For any class of functions $\cf$, recall the bracketing entropy $H_{[\hspace{0.1em}]}(u,\cf)=\log N_{[\hspace{0.1em}]}(u,\cf)$, with $N_{[\hspace{0.1em}]}(u,\cf)$ the bracketing number under the Hellinger metric.  Useful will be the inequality $H_{[\hspace{0.1em}]}(u,\cf)\leq H_\infty(u/2,\cf)$, see Lemma~2.1 in \cite{vandeGeer2000}, where $H_\infty(u,\cf)=\log N_\infty(u,\cf)$, with $N_\infty(u,\cf)$ the covering number of $\cf$ with balls of radius $u$ under the supremum norm. For the latter we have the following result.

\begin{lemma}\label{lemma:covering}
Let $\cf_n$ be the set of probability measures $\mathbb{Q}_v$, where the L\'evy densities $v$ are elements of $V_n$.
It holds that $H_\infty(u,\cf_n)\asymp N_n \log (1+\frac{1}{u})$, and hence there is $C>0$ such that for all sufficiently small $\delta>0$ and sufficiently large $N_n$ (the number of bins), one has
\[
\int_0^\delta H_\infty^{1/2}(u,\cf_n)\,\dd u\leq C \sqrt{N_n}\delta\log^{1/2} \left(\frac{1}{\delta}+1\right). 
\]
\end{lemma}

\begin{proof}
Starting point is the result of Lemma~\ref{lemma:hell}. First we need a $\delta$-cover of the interval $[\underline{\alpha},\overline{\alpha}]$, for which the covering number needed is of order $\delta^{-1}+1$. To cover a set of functions $\theta$, it is sufficient to cover the bounded intervals to which the corresponding $\rho_k$ and $\theta_k$ belong. Hence $\delta$-covers for both are again of order $\delta^{-1}+1$, and we have to do this on $N_n$ bins separately.
Altogether, this implies that a cover of size $O(\delta^{-1}+1)^{2N_n+1}$ is sufficient to cover the set $\cf_n$. 
Hence $\int_0^\delta H_\infty^{1/2}(u,\cf_n)\,\dd u\asymp \sqrt{N_n}\int_0^\delta \log^{1/2} (u^{-1}+1)\,\dd u$. We now show that the latter integral is of order $\delta\log^{1/2}(1+\frac{1}{\delta})$ for small $\delta$. For this we assume that $\delta<\frac{1}{e-1}$, which entails $\log(1+\frac{1}{\delta})> 1 >\frac{1}{1+\delta}$, $\log(y+1)>1$ and and $\frac{1}{y}<\frac{2}{y+1}$ for $y>\delta^{-1}$. These inequalities are used to show via lengthy but standard computations that
\[
\int_0^\delta \log^{1/2} (u^{-1}+1)\,\dd u \leq  2\delta\log^{1/2} (\delta^{-1}+1).
\]
The result of the lemma follows.
\end{proof}

The next result is used to handle the numerator in Bayes' formula in our main result, Theorem \ref{thm:consistency}.

\begin{lemma}
\label{lem:maximal}
Fix $\epsilon>0$ and define $B(\epsilon)=\{ v\in V_n\colon d_{\mathcal{H}}(\mathbb{Q}_{v_0},\mathbb{Q}_v) \leq \epsilon \}.$ Then there exist positive constants $c_1,c_2,c_3,$ independent of $n$, such that
\[
\mathbb{Q}_{v_0}^n\left( \sup_{v\in B(\epsilon)^c}  \prod_{i=1}^n \frac{\dd\mathbb{Q}_{v}}{\dd\mathbb{Q}_{v_0}}(Z_i) \geq \exp(-c_1n\epsilon^2) \right) \leq c_3 \exp(-c_2 n\epsilon^2).
\]
\end{lemma}

\begin{proof}
We will use Theorem~1 in \cite{wong95}.  The main fact to establish is a bound on the entropy integral $\int_0^\epsilon H_{[\hspace{0.1em}]}^{1/2}(u,\cf_n)\dd u$ (the set $\cf_n$ as in Lemma~\ref{lemma:covering}) of the form $C\sqrt{n}\epsilon^2$. It follows from Lemma~\ref{lemma:covering} and the remarks preceding it, that $\int_0^\epsilon H_{[\hspace{0.1em}]}^{1/2}(u,\cf_n)\dd u\leq C \sqrt{N_n}\epsilon\log^{1/2} (\frac{1}{\epsilon}+1)$. We want to choose $N_n$, so that
\[
\sqrt{N_n}\epsilon\log^{1/2} \left(\frac{1}{\epsilon}+1\right) \lesssim \sqrt{n}\epsilon^2
\]
for all $n$ and all small enough $\epsilon$. To that end it is enough to have
\[
\frac{N_n}{n} \lesssim \frac{\epsilon^2}{ \log (\frac{1}{\epsilon}+1) },
\]
which in fact holds for all $n$ large enough, since $N_n/n\rightarrow 0$ by assumption. Then Condition~(3.1) in \cite{wong95} is satisfied, and hence we can apply Theorem~1 of that paper, which yields the assertion.
\end{proof}

\section{Technical lemma for Section \ref{sec:beta}}
\label{lemmata}

\begin{lemma}\label{lem:globalbalance}
Let $I$ be a countable index set 
and $(E_i, \mathfrak A_i, \pp_i)$, $i \in I$, a collection of probability spaces or  $\sigma$-finite measure spaces.
Denote the corresponding product measurable space with the product measure by $(E, \mathfrak A, \pp)$. Let  $\pi^J\colon x \in E \mapsto (x_i)_{i \in J}$ be the coordinate projections for $J \subset I$.
Assume that $\QQ(x, \dd x^\circ)$ is a $\sigma$-finite transition measure with a localisation property
\[
\QQ(x; \dd \pi^{I_1 \cup I_2} (\, \cdot \,)) = \QQ(\pi^{I_1}(x); \dd \pi^{I_1}(\, \cdot \,)) \otimes  \QQ(\pi^{I_2}(x); \dd \pi^{I_2}(\, \cdot \,))
\]
for all $x \in E$, $I_1, I_2 \subset I$, $I_1 \cap I_2 = \varnothing$.
Then the local balance condition
\[
\pp^i(\dd x_i)\QQ^i(x_i; \dd x_i^\circ) = \pp^i(\dd x^\circ_i)\QQ^i(x_i^\circ; \dd x_i),
\]
where $\QQ^i(x_i; A) = \QQ(\pi^{i}(x); (\pi^i)^{-1}(A))$ for $A \in \mathfrak A_i$,
implies 
\begin{equation}\label{globalbalance}
\pp(\dd x)\QQ(x; \dd x^\circ) = \pp(\dd x^\circ)\QQ(x^\circ; \dd x).
\end{equation}
\end{lemma}
\begin{proof}
A measure on $E^2$ can be written as a measure on 
$\tilde E^2 = \bigtimes_{i \in \mathbb N} E^2_i$ using an obvious change of coordinates.
Denote the measure $\pp(\dd x)\QQ(x, \dd x^\circ)$  seen as a measure on $\tilde E^2$ by $\mu$. Then 
\[
\mu\left((\bigtimes_{i \le n} (A_i \times A^\circ_i)) \times (\bigtimes_{i > n} E^2_i)\right) = \prod_{i \le n} \int_{A_i} \QQ^i(x_i, A^\circ_i)\pp^i(\dd x), \quad A_i, A_i^\circ \in \mathfrak A_i
\]
for all $n \in \mathbb N$. Therefore $\mu$ is a product measure. It is also a symmetric measure in the following sense: $\mu(s(A)) = \mu(A)$ for $s(A) = \{ (x^\circ_i,x_i)_{i \in I} \colon (x_i,x^\circ_i)_{i \in I} \in A\}$. 
This can be formally shown by the ``good set principle'':
Let $\mathfrak S$ be the collection of sets such that  $\mu(s(S)) = \mu(S)$ holds for $S \in \mathfrak S$. First, $\bigtimes_{i \le n} (A_i \times A^\circ_i) \times (\bigtimes_{i > n} E^2_i)) \in \mathfrak S$, so $\mathfrak S$ contains a generator which has the intersection property ($\pi$-system). 
Now $E \in \mathfrak S$, and also complements of sets in $\mathfrak S$ are in $\mathfrak S$, and countable unions of  disjoint sets in $\mathfrak S$ are in $\mathfrak S$ as well: if $A_i \in \mathfrak S$ are disjoint sets and $A = \bigcup A_i$, then
\[
\mu(A) = \sum \mu(A_i) = \sum \mu(s(A_i)) = \mu(s(A)).
\]
Therefore $\mathfrak S = \mathfrak A$ by Dynkin's $\pi$-$\lambda$ theorem.
The balance equation \eqref{globalbalance} follows.
\end{proof}

\section{Danish fire losses: exploratory data analysis}
\label{app:danish}

In this appendix we perform an exploratory analysis of the Danish data on large fire losses. We primarily use graphical tools; these may look simple, but are commonly applied in similar analyses (see, e.g., \cite{mcneil97} and \cite{resnick97}) and convey useful information that is not easily obtainable otherwise.

Figure \ref{fig:danishdata:acf:pacf} gives the plots of the estimated autocorrelation and partial autocorrelation functions of logarithmically transformed and aggregated Danish fire losses. Both plots are compatible with the assumption that the data follow a white noise process. A more formal confirmation comes from the Box-Pierce and Ljung-Box tests, that we applied with 20 lags, and that yielded $p$-values $0.5847$ and $0.5547$, respectively (the tests are implemented in {\bf R} via {\tt{Box.test}}). This suggests that weekly data can indeed be modelled as an i.i.d.\ sequence.

\begin{figure}
\includegraphics[width=0.45\textwidth]{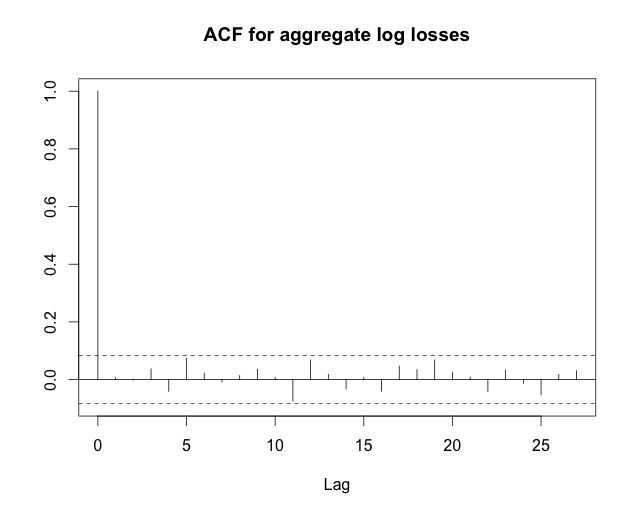}
\includegraphics[width=0.45\textwidth]{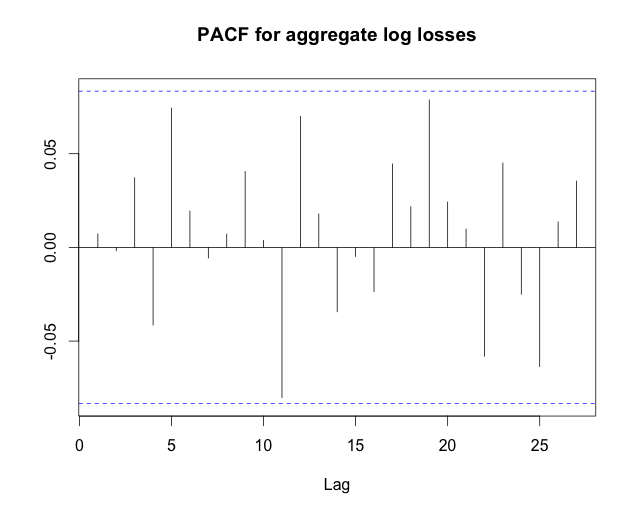}
\caption{Logarithmically transformed and aggregated Danish data on large fire losses. {\em Left}: autocorrelation function. {\em Right}: partial autocorrelation function.}
\label{fig:danishdata:acf:pacf}
\end{figure}

We also produced the histogram of the weekly data, and fitted the Gamma distribution via the maximum likelihood method. The results are displayed in the left panel of Figure \ref{fig:histo_qq}, and provide a visual hint that a Gamma-type distribution yields a reasonable fit to the data. Since a histogram is a somewhat crude nonparametric estimator and is strongly dependent on the choice of the bin number (we used the default implementation in {\bf R} via the command {\tt{hist}}), we also visually compared the Gamma fit to a kernel density estimator, with bandwidth selected through cross-validation (we used the {\tt{density}} in {\bf R} with the Gaussian kernel), see the right panel of Figure \ref{fig:histo_qq}. Ignoring the edge effects near the boundary point of the support of the distribution, it appears that the two estimates are different e.g. in a neighbourhood of the mode of the Gamma density, with probability mass of the kernel density estimate shifted to the right.
On the other hand, the tail behaviour of both estimates is similar.

Although evidence is not decisive, a further hint that the Gamma distribution is perhaps not entirely adequate for modelling the Danish fire losses data comes from the QQ-plot of empirical quantiles of the Danish fire losses data versus theoretical Gamma quantiles; see Figure \ref{fig:qq} (we used the command {\tt{qqPlot}} from the {\tt{car}} package in {\bf R}).
\begin{figure}
\includegraphics[width=0.45\textwidth]{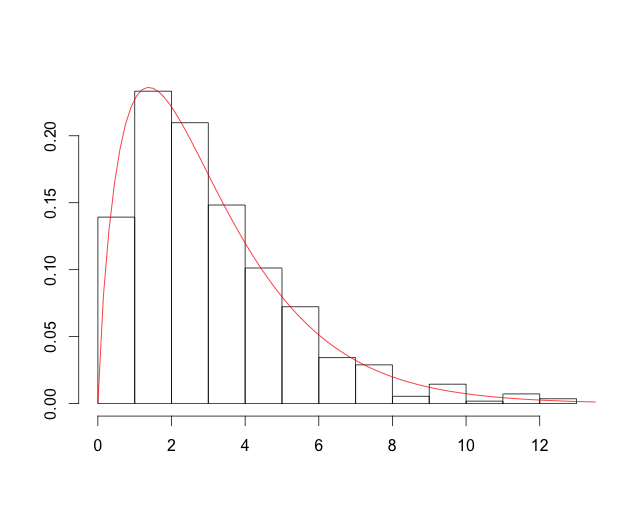}
\includegraphics[width=0.45\textwidth]{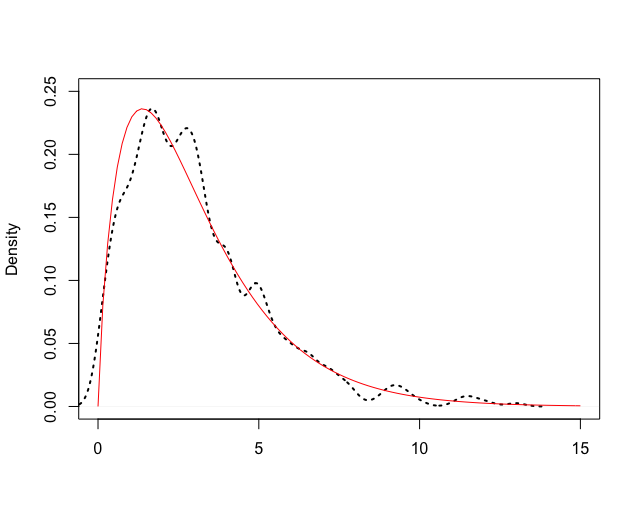}
\caption{Logarithmically transformed and aggregated Danish data on large fire losses. {\em Left}: histogram with a superimposed gamma density evaluated at the maximum likelihood estimate. {\em Right}: kernel density estimate (dotted line) with the same superimposed gamma density (solid line) evaluated at the maximum likelihood estimate.}
\label{fig:histo_qq}
\end{figure}
\begin{figure}
\includegraphics[width=0.75\textwidth]{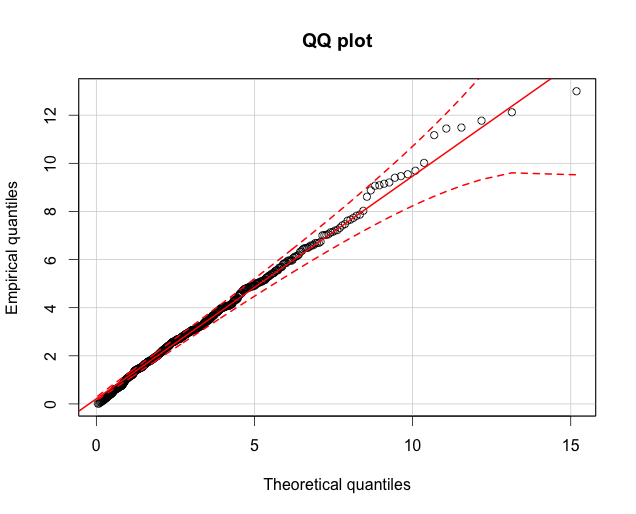}
\caption{Logarithmically transformed and aggregated Danish data on large fire losses: QQ-plot of empirical quantiles versus theoretical gamma quantiles.}
\label{fig:qq}
\end{figure}

Summarising the results of our exploratory data analysis, it appears that if aggregated over weekly (or in some exceptional cases over bi-weekly) periods, the logarithmically transformed Danish fire losses data can be adequately modelled as a realisation of an i.i.d.\ sequence that follows a Gamma-like distribution, but perhaps is not genuinely  Gamma.

\par{\bf Acknowledgement.}\,
The research leading to the results in this paper has received funding from the European Research Council under ERC Grant Agreement 320637. The research  of the first author was supported by the Russian Academic Excellence Project ``5-100'' and the German Science Foundation  research grant (DFG Sachbeihilfe) 406700014. The authors are grateful to anonymous referees for careful reading and insightful comments that lead to improvements in the paper.

\vskip2mm

\bibliographystyle{plainnat}

\end{document}